\documentclass{amsart}
\usepackage{amsfonts,amscd,amssymb,amsmath,amsthm}
\usepackage{graphicx,mathrsfs,appendix}
\usepackage{xparse,centernot}
\usepackage[T1]{fontenc}

\usepackage{tikz}

\begin{document}

\newtheorem{theorem}{Theorem}[section]
\newtheorem*{theorem*}{Theorem}
\newtheorem{lemma}[theorem]{Lemma}
\newtheorem{corollary}[theorem]{Corollary}
\newtheorem{conjecture}[theorem]{Conjecture}
\newtheorem{cor}[theorem]{Corollary}
\newtheorem{proposition}[theorem]{Proposition}
\theoremstyle{definition}
\newtheorem{definition}[theorem]{Definition}
\newtheorem{example}[theorem]{Example}
\newtheorem{claim}[theorem]{Claim}
\newtheorem{remark}[theorem]{Remark}

\newenvironment{pfofthm}[1]
{\par\vskip2\parsep\noindent{\sc Proof of\ #1. }}{{\hfill
$\Box$}
\par\vskip2\parsep}
\newenvironment{pfoflem}[1]
{\par\vskip2\parsep\noindent{\sc Proof of Lemma\ #1. }}{{\hfill
$\Box$}
\par\vskip2\parsep}

%%%%%%%%%%%%%%%%%%%%%%%%%%%%%%%%%%%%%%%%%%
%%% General macros
%%%%%%%%%%%%%%%%%%%%%%%%%%%%%%%%%%%%%%%%%%

\newcommand{\R}{\mathbb{R}}
\newcommand{\T}{\mathcal{T}}
\newcommand{\U}{\mathcal{U}}
\newcommand{\Z}{\mathbb{Z}}
\newcommand{\Q}{\mathbb{Q}}
\newcommand{\E}{\mathbb E}
\newcommand{\N}{\mathbb N}
\newcommand{\config}{{\bf C}}
\newcommand{\barray}{\begin{eqnarray*}}
\newcommand{\earray}{\end{eqnarray*}}

\newcommand{\beq}{\begin{equation}}
\newcommand{\eeq}{\end{equation}}

%%%%%%%%%%%%%%%%%%%%%%%%%%%%%%%%%%%%%%%%%%
%%% Probability Macros
%%%%%%%%%%%%%%%%%%%%%%%%%%%%%%%%%%%%%%%%%%

\newcommand{\Prob}{\mathbb{P}}
% I need lots of redundant macros here
\newcommand{\Exp}{\mathbb{E}}
\newcommand{\expect}{\mathbb{E}}
\newcommand{\1}{\mathbf{1}}
\newcommand{\prob}{\mathbb{P}}
\newcommand{\pr}{\mathbb P}
\newcommand{\filt}{\mathscr{F}}
\DeclareDocumentCommand \one { o }
{%
\IfNoValueTF {#1}
{\mathbf{1}  }
{\mathbf{1}\left\{ {#1} \right\} }%
}
\newcommand{\Bernoulli}{\operatorname{Bernoulli}}
\newcommand{\Binomial}{\operatorname{Binom}}
\newcommand{\Binom}{\Binomial}
\newcommand{\Poisson}{\operatorname{Poisson}}
\newcommand{\Exponential}{\operatorname{Exp}}
\newcommand{\weakto}{\Rightarrow}

%%%%%%%%%%%%%%%%%%%%%%%%%%%%%%%%%%%%%%%%%%
%%% Random Graph/Complex Macros
%%%%%%%%%%%%%%%%%%%%%%%%%%%%%%%%%%%%%%%%%%

\newcommand{\link}{\mbox{lk}}
\newcommand{\Deg}{\operatorname{deg}}
\newcommand{\vertexsetof}{\mathcal{V}}
\renewcommand{\deg}{\Deg}
\newcommand{\oneE}[2]{\mathbf{1}_{#1 \leftrightarrow #2}}
\newcommand{\ebetween}[2]{{#1} \leftrightarrow {#2}}
\newcommand{\noebetween}[2]{{#1} \centernot{\leftrightarrow} {#2}}
\newcommand{\Gap}{\ensuremath{\tilde \lambda_2 \vee |\tilde \lambda_n|}}
\newcommand{\dset}[2]{\ensuremath{ e({#1},{#2})}}
\newcommand{\EL}{{ L}}
\newcommand{\ER}{{Erd\H{o}s--R\'{e}nyi }}
\newcommand{\zuk}{{\.{Z}uk}}
\newcommand{\BS}{{Ballman--\'Swi\k{a}tkowski }}

%%%%%%%%%%%%%%%%%%%%%%%%%%%%%%%%%%%%%%%%%%
%%% Paper-Specific Macros
%%%%%%%%%%%%%%%%%%%%%%%%%%%%%%%%%%%%%%%%%%

\newcommand{\bdc}{\ensuremath{ \bf{b.d.c.}}}
\DeclareDocumentCommand \fuzz { o o }
{
\IfNoValueTF {#1}
  { \aleph_M }
  { \IfNoValueTF { #2 }
    { \aleph_{M}^{{#1}} }
    { \aleph_{M}^{{#1}}({#2})}
  }
}
\newcommand{\csubzero}{c_{0}}
\newcommand{\csubone}{c_{1}}
\newcommand{\csubtwo}{c_{2}}
\newcommand{\csubthree}{c_{3}}
\newcommand{\csubstar}{c_{*}}
\newcommand{\rsp}{1-C\exp(-md^{1/4}\log n)}
\newcommand{\lc}{\ensuremath{ \operatorname{light}(x,y)}}
\newcommand{\hc}{\ensuremath{ \operatorname{heavy}(x,y)}}
\DeclareDocumentCommand \LMP { O{t} O{d} }
{
Y^{{#2}}_{{#1}}(n)
}
\DeclareDocumentCommand \tLMP { O{t} O{d} }
{
\tilde{Y}^{{#2}}_{{#1}}(n)
}
\newcommand{\cd}{\operatorname{cd}}
\newtheorem*{LMW}{Linial--Meshulam--Wallach theorem}
\newtheorem*{Zuk}{\.Zuk's criterion}
\newtheorem*{LM}{Linial--Meshulam theorem}
\newtheorem*{BSC}{\BS criterion}

\title[Spectral gaps of random graphs]{Spectral gaps of random graphs and applications}
\author{Christopher Hoffman}
\address{Department of Mathematics, University of Washington}
\email{hoffman@math.washington.edu}
\author{Matthew Kahle}
\address{Department of Mathematics, The Ohio State University}
\email{mkahle@math.osu.edu}
\author{Elliot Paquette}
\address{Department of Mathematics, The Ohio State University}
\email{paquette.30@osu.edu}
\thanks{The first author was supported by an AMS Centennial Fellowship and NSF grant DMS-0806024. The second author was supported by NSA grant \# H98230-10-1-0227, DARPA grant \#N66001-12-1-4226, NSF grant \#DMS-1352386, and an Alfred P.\ Sloan research fellowship. The third author was supported by NSF grants DMS-0847661 and DMS-0806024 and NSF postdoctoral fellowship DMS-1304057.}
\date{\today}
\maketitle

%\begin{abstract}  We show that $2 \log{n} / n$ is the threshold for the fundamental group of random $2$-complexes to have property~(T) with high probability.  This coincides with the vanishing threshold for cohomology with finite field coefficients found earlier by Linial, Meshulam, and Wallach.  As an intermediate result, we establish new concentration results for spectral gaps of \ER random graphs.
%\end{abstract}

\begin{abstract}
%The spectrum of the \ER  graph $G(n,p)$ has been the subject of intensive study, starting with Furedi and Komlos \cite{FurediKomlos} and continuing through the present \cite{CojaOghlan} \cite{FeigeOfek} \cite{ChungLuVu}. --FK don't actually mention ER
We study the spectral gap of the  \ER random graph through the connectivity threshold. In particular, we show that
for any fixed $\delta > 0$
if $$p \ge \frac{(1/2 + \delta) \log n}{n},$$
then the normalized graph Laplacian of an \ER graph has all of its nonzero eigenvalues tightly concentrated around $1$. 
%In particular, they are all concentrated around $1$ near the connectivity threshold of $p = \log n / n $. 
This is a strong expander property.  

We estimate both the decay rate of the spectral gap to $1$ and the failure probability, up to a constant factor. We also show that the $1/2$ in the above is optimal, and that if $p = \frac{c \log n}{n}$ for $c < 1/2,$ then there are eigenvalues of the Laplacian restricted to the giant component that are separated from $1.$

We then describe several applications of our spectral gap results to stochastic topology and geometric group theory. These all depend on Garland's method \cite{Garland}, a kind of spectral geometry for simplicial complexes. The following can all be considered to be higher-dimensional expander properties.

First, we exhibit a sharp threshold for the fundamental group of the Bernoulli random $2$-complex to have Kazhdan's property (T). We also obtain slightly more information and can describe the large-scale structure of the group just before the (T) threshold. In this regime, the random fundamental group is with high probability the free product of a (T) group with a free group, where the free group has one generator for every isolated edge. The (T) group plays a role analogous to that of a ``giant component'' in percolation theory.

Next, we give a new, short, self-contained proof of the Linial--Meshulam--Wallach theorem \cite{LM06,MW09}, identifying the cohomology-vanishing threshold of Bernoulli random $d$-complexes. Since we use spectral techniques, it only holds for $\Q$ or $\R$ coefficients rather than finite field coefficients, as in \cite{LM06} and \cite{MW09}. But it is sharp from a probabilistic point of view, providing for example, hitting time type results and limiting Poisson distributions inside the critical window. It is also a new method of proof, circumventing the combinatorial complications of cocycle counting. Similarly, results in an earlier preprint version of this article were already applied in \cite{flag} to obtain sharp cohomology-vanishing thresholds in every dimension for the random flag complex model.
%
% As a corollary, all the cohomology of these complexes is usually concentrated in one dimension. In a similar application, the results from this article were used in \cite{DK12} to show that random right-angled Coxeter groups built on random graphs are ``virtual duality groups'' and in particular have all their group cohomology concentrated in one dimension.
\end{abstract}

\section{Introduction}

Studying the spectral properties of random matrices has played a central role in probability theory ever since Wigner's paper establishing the semi-circular law for symmetric matrices with independent entries of equal variance~\cite{wigner}.  The theory of these matrices is rich and well-developed, and its techniques and theorems provide great insight into the adjacency matrices of random graphs.
%Furedi and Komlos \cite{FurediKomlos} started the study of the support of the spectrum of such a matrix, of an \ER  random graph and their work has been extended by many others %\cite{CojaOghlan} \cite{FeigeOfek} \cite{ChungLuVu}. %\cite{ChungRadcliffe}.

In this paper we study the normalized Laplacian matrix of a Bernoulli (also \ER) random graph $G(n,p),$ which has $n$ vertices and whose every edge is included independently with probability $p=p(n).$  For a connected graph $G,$ the normalized Laplacian has smallest eigenvalue $\lambda_1=0,$ and the remainder of its eigenvalues $\{\lambda_i\}_{i=2}^n$ lie in the interval $0 < \lambda_i \le 2$.  The spectral gap, $\lambda_2,$ is the principal quantity of interest in many applications, and it has received much attention in the literature~\cite{Chung,ChungLuVu,ChungRadcliffe,CojaOghlan}.  

Our focus is on typical behavior of random graphs for large values of $n.$  So, we will use the terminology \emph{with high probability} (abbreviated w.h.p.) as a qualifier for a statement holds with probability tending to $1$ as $n$ tends to infinity.  We will also use the expression \emph{with overwhelming probability}, meaning the statement holds with failure probability smaller than $O(n^{-C})$ for all $C>0.$  

We will make use of the Landau notations $O, o, \omega, \Omega, \Theta$ in the asymptotic sense, so that $f=O(g)$ means $f/g$ is eventually bounded above as $n \to \infty$ and $f=o(g)$ means $f/g$ tends to $0$ as $n \to \infty.$  Also, $f=\omega(g)$ means $g=o(f)$ and $f=\Omega(g)$ means $g=O(f).$ Finally, we will use $f=\Theta(g)$ to mean $f=O(g)$ and $f=\Omega(g).$

We will also make use of the notion of \emph{thresholds}.  A function $f=f(n)$ is said to be a threshold for a property $\mathcal{P}$ if $p = \omega(f)$ implies $G \in \mathcal{P}$ w.h.p.\ and $p=o(f)$ implies $G \not\in \mathcal{P}$ w.h.p.\  Such a threshold is only defined up to $n$--independent scalar multiples.  If there is a function $g=o(f)$ so that $p \geq f + g$ implies $G \in \mathcal{P}$ w.h.p.\ and $p \leq f - g$ implies $G \not\in \mathcal{P}$ w.h.p.\ the threshold is \emph{sharp}.  If no such $g$ exists, the threshold is \emph{coarse}.  

A fundamental result of random graph theory is that every nontrivial monotone property has a threshold \cite{friedgutkalai}, which need not be sharp.  For example, the appearance of triangles in $G(n,p)$ has the threshold $1/n,$ which is coarse.  On the other hand, the \ER theorem shows that $\log n / n$ is the sharp threshold for connectivity of the graph.  Similarily, we will need that $\tfrac{1}{2} \log n / n$ is the sharp threshold for the graph to consist only of one \emph{giant component} $\tilde G$ and isolated vertices, which is an easy extension of the \ER theorem.  We will use $\tilde G$ to denote the largest connected component of $G(n,p),$ which is well--defined w.h.p.\ for $p = \omega(1/n)$ (see \cite{JKLP} for a detailed discussion or Lemma \ref{lem:giant_component_size}).

For the \ER graph, as we shall show, the eigenvalues $\{\lambda_i\}_{i=2}^n$ tend to cluster around $1,$ and hence we define $\lambda(G)=\max_{i \neq 1}|1-\lambda_i|.$  The quantity $1-\lambda(G)$ is sometimes referred to as the \emph{absolute gap.}
%(Precise definitions are given in later sections.)
The methods in the previous papers are successful in establishing the correct order for $\lambda(G)$ of $C(np)^{-1/2}$ when the density of edges is sufficiently large,
%(see Section~\ref{sec:rmt_background} for more on why this is correct)
but they do not extend to $p$ very near the connectivity threshold $\log n / n$.  

Our main result on spectral gaps are contained in the following two theorems.  
\begin{theorem}
\label{thm:ergap_giant}
Fix $\delta > 0$ and let
\(p \geq (\frac12+\delta) \log n / n\).
Let $d=p(n-1)$ denote the expected degree of a vertex.
%Let $\tilde G$ be the giant component of the \ER graph.
For every fixed $\epsilon > 0,$
there is a constant $C=C(\delta, \epsilon),$
so that
\[
\lambda(\tilde G)< \frac{C}{\sqrt{d}}.
\]
with probability at least $1-Cn\exp(-(2-\epsilon)d)-C\exp(-d^{1/4}\log n).$
\end{theorem}

%As a corollary, it follows immediately that we get:
%\begin{corollary}
%\label{thm:ergap_connected}
%Fix $\omega(n) \to \infty.$
%If
%\[
%p \geq \frac{(m+1) \log n + \omega(n)}{n},
%\]
%then there is a constant $C=C(m)$ so that with probability $1-o(n^{-m}),$ we have
%\[
%\max_{i > 1} \left| 1 - \lambda_i\right| < \frac{C}{\sqrt{d}}.
%\]
%Conversely, if
%\[
%p \leq \frac{(m+1) \log n - \omega(n)}{n},
%\]
%then
%\[
%\max_{i > 1} \left| 1 - \lambda_i\right| = 1
%\]
%with probability $\Omega(n^{-m}).$
%\end{corollary}

This result improves on a number of previous results. These earlier results are discussed in more detail in Section \ref{sec:rmt_background}. In brief, the state of the art is due to Coja--Oghlan \cite{CojaOghlan} who obtains gap $1-O(d^{-1/2})$ for $p \ge C \log n / n$, where $C> 0$ is a sufficiently large constant.
We are able to extend this to $C=1,$ and appropriately modifying the statement for the giant component, we extend this to $C = \tfrac12$.

We note that Theorem \ref{thm:ergap_giant} is vacuous for $p \leq \frac{1}{2} \log n / n.$
Indeed, the next result shows that for smaller values of $p,$ the gap is no longer
$1-o(1).$
\begin{theorem}
\label{thm:ergap_giant2}
For $p$ satisfying $p = \omega(\sqrt{\log n}/n)$ and $p \leq \tfrac 12 \log n / n$
\[
\lambda(\tilde G) \geq \tfrac 12,
\]
with high probability.
\end{theorem}

For $p = O( \sqrt{\log n}/n),$ Fountoulakis and Reed~\cite{FR} show that the mixing time is large, and hence provide a lower bound for $\lambda(\tilde G)$ in this regime.
So $G(n,p)$ has $\lambda(\tilde G)$ bounded away from $0$, but at $ \frac{1}{2} \log n / n$ there is a phase transition, and at this point $\lambda(\tilde G) = o(1)$.  
We in fact prove a slightly stronger result than Theorem~\ref{thm:ergap_giant2} in Section~\ref{sec:pbnds} (c.f.\ Lemma~\ref{lem:deterministic2}).

We also consider an \ER process version (see Section \ref{sec:proc} for definitions) of the spectral gap theorem. In particular, we show that if random edges are added one at a time, at the moment of connectivity the random graph already has spectral gap $1-o(1)$. More precisely, we have the following.

\begin{theorem} \label{thm:erstop}
Let $\tau_c$ be the connection time for the \ER graph process $G(n,m).$  Then there is a constant $C$ so that with high probability
\[
\lambda(G(n,\tau_c)) \leq C/\sqrt{\log n}.
\]
\end{theorem}

This theorem follows immediately from Theorem~\ref{thm:stopping_time} in the $2$--dimensional case: that theorem shows that the largest component of the \ER graph process has gap $\lambda( \tilde G) \leq C/\sqrt{\log n}$ for all time after $(\tfrac 14 +\delta)n\log n$ edges have been added, w.h.p.  Hence, at the connection time $\tau_c$, which occurs when about $(\tfrac 12)n \log n$ edges have been added, $ \lambda( G )  = \lambda( \tilde G) \leq C/\sqrt{\log n}.$
%, which contains a stronger result about the \ER process valid for the spectral gap of the giant component.

\section*{Applications to stochastic topology}

As we will see, Theorem \ref{thm:ergap_giant} is useful in the study of random topological spaces and random groups. We now provide several examples where this theorem yields sharp results. All of these new results depend on the  combination of the spectral gap theorem with ``Garland's method'' and its refinements by Ballman and \'Swi\k{a}tkowski \cite{BS}, and by \.Zuk \cite{Zuk,zuk2}.

%By optimizing the spectral gap theorem, we also optimize the high-dimensional applications. All of the following applications depend on our new spectral gap theorem. Some of these results are new, and some are regine

%\subsection{Kazhdan's property (T)}
\vspace{.05in}\noindent{\bf $\bullet$ Kazhdan's property (T).}
Linial and Meshulam~\cite{LM06} introduce an analogous measure $Y_2(n,p)$ to the binomial random graph for random $2$-dimensional simplicial complexes.
This is the probability distribution on all simplicial complexes with vertex set $[n]=\{1,2,\dots,n\}$, with complete $1$-skeleton (i.e.\ with all possible $n \choose 2$ edges), and such that each of the ${n \choose 3}$ possible $2$-dimensional faces are included independently with probability $p$. We use the notation $Y \sim Y_2(n,p)$ to indicate a complex drawn from this distribution.  We will call an edge \emph{isolated} if no triangle contains it.

We prove here a structure theorem for the random fundamental group, for a certain range of $p$.

\begin{theorem} \label{thm:structure}
  Suppose $\delta > 0$ is fixed,
$$ p \ge \frac{ ( 1 + \delta) \log n}{n},$$
and $Y \sim Y_2(n,p)$.
Then w.h.p.\ $\pi_1(Y)$ is isomorphic to the free product of a (T) group $G$, and a free group $F$, where the free group $F$ has one generator for every isolated edge in $Y$.
\end{theorem}

As a corollary, we also show that the threshold for $\pi_1(Y)$ to have property (T) agrees precisely with the homology-vanishing threshold found by Linial and Meshulam \cite{LM06}.  For the proof, along with further details and explanation, see Section \ref{sec:pi1}. %This is true in the strong ``hitting time'' sense.

It might be that $\pi_1(Y)$ is a free product of a (T) group and a free group for smaller $p$. The most interesting conjecture about the structure of $\pi_1(Y)$ might be the birth of a giant (T) free factor at $p = c / n$ for some constant $c \approx 2.7538$. This is the same point as the homological phase transition studied by Linial and Peled \cite{LP16}.

\vspace{.05in}\noindent{\bf $\bullet$ Random $d$-dimensional simplicial complexes.}
Meshulam and Wallach further generalize the $2$-dimensional model to random $d$-dimensional complexes $Y_d(n,p)$ \cite{MW09}. Their main result is that $ p = d \log n / n$ is a sharp threshold for vanishing of cohomology $H^{d-1}(Y, {\bf k})$ where ${\bf k}$ is a finite field or a field of characteristic $0$. The proof requires delicate cocycle counting arguments.

The new spectral gap results give a new proof of the Meshulam--Wallach theorem, in the case that $k$ is a field of characteristic $0$. The Meshulam--Wallach theorem is stronger topologically, since homology vanishing over a finite field implies vanishing over $\Q.$ But our new proof is very short (given the spectral gap theorem), and the result is slightly sharper probabilistically. For example, we obtain hitting time results in an accompanying stochastic growth process (see for Corollaries \ref{cor:Hhit} and \ref{cor:Thit} for representative examples of ``hitting time results''), and also we recover a simple proof of the Poisson distribution of Betti numbers in the critical window (Corollary \ref{cor:HPoisson}).

Gundert and Wagner show that the Laplacian on $(d-1)$-forms in a random $d$-complex has a large spectral gap for $p \ge C_d \log n / n$ for some sufficiently large $C_d$ \cite{GW12}. Combining their argument with the results in this paper would yield a hitting time result, and in particular this shows that the gap for these higher Laplacians is already large for $p \ge d \log n / n$.

Parzanchevski, Rosenthal, Tessler \cite{PRT13} combine Gundert and Wagner's argument with earlier work of Pach \cite{Pach98} to show that for $p \ge C_d \log n / n$, w.h.p.\ $Y$ has the ``geometric overlap'' property.\footnote{A sequence of $d$-dimensional simplicial complexes $S_n$ with $F_n$ $d$-dimensional faces has the \emph{geometric overlap} property if there exists a constant $\lambda > 0$ so that for every geometric map: $S_n \to \R^d$ (i.e.\ affine linear on each face), there exists a point $p \in \R^d$ that lies in the image of at least $\lambda F_n$ $d$-faces. See for example recent work by Gromov and collaborators in \cite{Gromov09}, \cite{Gromov10}, and \cite{FGLNP12}.} It also seems possible to use the new spectral gap results to sharpen this result, and show that in the process version of the random complex, random $d$-complexes already have the geometric overlap property as soon as they are pure $d$-dimensional.

As far as we can tell, these suggested sharpenings of the main theorems in \cite{GW12} and \cite{PRT13} are not written down anywhere, and we do not further elaborate on them in this note. It seems that these sharper results only depend on substituting our Theorem \ref{thm:ergap_giant} for earlier results on spectral gap $G(n,p)$. 

\vspace{.05in}\noindent{\bf $\bullet$ Triangular random groups.}
Antoniuk et.\ al.\ study the phase transitions that occur in the triangular model of random groups \cite{densitymodel}. Similarly, by using our spectral gap results, their results can be strengthened, for example to show a hitting time result.

\vspace{.05in}\noindent{\bf $\bullet$ Random flag complexes.}
Let $X(n,p)$ denote the random clique complex, i.e. the maximal simplicial complex, with respect to inclusion of faces, whose $1$--skeleton is given by an \ER graph $G(n,p).$

Combining the spectral gap theorem from an earlier version of this paper with Garland's method, similar cohomology vanishing results were recently obtained for $X(n,p)$ by the second author in \cite{flag}. Combining with several earlier results \cite{clique}, as a corollary this shows that for every $d \ge 3$, there is a wide range of $p$ for which $X(n,p)$ is rationally homotopy equivalent to a bouquet of $d$-dimensional spheres.\footnote{A simplicial complex is \emph{rationally homotopy equivalent} to a bouquet of $d$ spheres if it is simply connected and all of its nontrivial reduced, rational, homology is in degree $d$.}

\vspace{.05in}\noindent{\bf $\bullet$ Random right-angled Coxeter groups.} Group cohomology of random right-angled Coxeter groups were studied in \cite{DK12}. Applying the same techniques as in the random flag complex paper \cite{flag}, it is shown that for a certain measure and range of parameter, random right-angled Coxeter groups are rational duality groups with high probability. This is actually a special case of a more general statement that shows that the same holds for random graph products of finite groups.

\section*{Organization}

Section \ref{sec:rmt_background} contains the background about the spectrum of the normalized Laplacian of \ER random graphs. Section \ref{sec:rt} does the same for our applications of our spectral results to random topology.
In Section \ref{sec:est} we show how to transfer adjacency matrix estimates to the normalized Laplacian under some assumptions on the structure of the graph.
In Section~\ref{sec:pbnds} we show that an \ER graph satisfies these structural conditions with high probability.
In Section~\ref{sec:proc} we show that the Linial-Meshulam process has large gap in a local spectral sense.
In Section~\ref{sec:cohom} we show how to apply the \BS criterion to prove the structure theorem for rational cohomology, and in Section~\ref{sec:propT} we show how to apply \.Zuk's criterion to prove the structure theorem for the fundamental group.
In Section~\ref{sec:KSz} we apply the Kahn-Szemer\'erdi machinery to show that the adjacency matrix of the \ER graph has a gap of the correct order for any $p$ with $p = \Omega( \log n / n).$
Finally, we include one appendix which proves the precise versions of the tail bounds for binomial variables that we use.

\section{Background: spectra of random graphs} \label{sec:rmt_background}

There are multiple common notions of spectra of a graph.  The most elementary definition is given by the eigenvalues of the adjacency matrix $A$.  The subjects of our main theorems are the eigenvalues of the normalized Laplacian $\EL$ (see~\eqref{eq:nldef} for a precise definition).  %The spectra of the normalized Laplacian are always one minus the spectra of the Markov operator associated to simple random walk on the graph.  
When the graph is regular, these two notions of spectra are just shifted rescalings of one another.

Appropriately, when the graph is nearly regular, as is the case for $G(n,p)$ with $p = \omega(\log n / n),$ these two spectra behave in nearly the same way.  Coarse statements about the {spectral gap} of $G(n,p)$ in this regime can largely be considered a statement about either spectra, and indeed, the primary method for estimating the gap of $\EL$ in the setting of \ER graphs is by comparison with $A.$

We will now give a precise definition of the normalized Laplacian.
A good general introduction to the properties of the normalized Laplacian is available in~\cite{Chung}.
Let $\pi_+$ be the projection map onto the vertices with positive degree, let $T$ be the diagonal matrix of degrees, and let $A$ be the adjacency matrix.
The normalized Laplacian is defined as
\begin{equation}\label{eq:nldef}
\EL = \pi_{+} - T^{-1/2}AT^{-1/2},
\end{equation}
where $T^{-1/2}$ is taken to be $0$ in coordinates where the degree is $0.$
Note that some authors use an alternate definition of normalized Laplacian, with $\pi_{+}$ replaced by $\operatorname{Id}$.  We let $0 =\lambda_1 \leq \lambda_2 \leq \ldots \leq \lambda_n \leq 2$ be the eigenvalues of $\EL.$

 The principal nontrivial property we will employ about $\EL$ is that the dimension of the kernel is equal to the number of components of $G$.  An immediate consequence is that for a graph with multiple nontrivial components, $\lambda_2=0.$ In particular, when $np - \log n \to -\infty$ the normalized Laplacian has no spectral gap with high probability. That said, it still makes sense to consider the spectral gap of $\EL$ restricted to the giant component.

\subsubsection*{Techniques for estimating eigenvalues}

As $A$ has i.i.d.\ entries above the diagonal, many off-the-shelf techniques can be applied to it directly.  In particular, the original \emph{trace method} bound of F\"uredi and Koml\'os~\cite{FurediKomlos} can be extended to show that when $p = \omega(\log^6 n/n),$ the second largest eigenvalue of the adjacency matrix of an \ER graph is of smaller order than the largest eigenvalue.  Improvements and corrections to this argument brought the bound to $p = \omega(\log^4 n /n),$ \cite{Vu07} and later to as low as $p \gg \log^2 n/n$ \cite{ChungLuVu}.  Newer methods have been pursued in \cite{llv}, \cite{bbk}, \cite{bbk2}.

The alternative method of
Kahn and Szemer\'edi~\cite{FKSz}, first developed for bounding the spectral gap of $d$-regular graphs, has been adapted quite successfully for estimating the spectral gap in the $p = \Theta(\log n / n)$ regime by Feige and Ofek~\cite{FeigeOfek}.  In particular, they show that there are constants $c > 0$ and $K > 0$ so that for $p > c\log n / n,$ all but the first eigenvalue are at most $K\sqrt{np}.$

One contribution of this paper is a sharpening of this estimate (see Proposition~\ref{prop:adj}).  Indeed, we show that for \emph{any} $c>0,$ there is a $K>0$ so that for $p > c\log n / n,$ all but the first eigenvalue are at most $K\sqrt{np}.$  Conversely, it is easily checked that for $p = o( \log n / n),$ there are many eigenvalues greater in magnitude than $\sqrt{np},$ coming from the existence of high-degree stars in the graph.  Thus, in a sense, we sharpen the Kahn-Szemer\'edi analysis of the \emph{full} adjacency matrix of $G(n,p)$ to its natural endpoint.

However, our main contribution in this paper is a technique for exactly characterizing \emph{when and why} the extremal eigenvalues of the normalized Laplacian stop tracking the extremal eigenvalues of the adjacency matrix.  Throughout the $p = \Theta(\log n /n),$ the extremal eigenvalues of the adjacency matrix do not undergo a phase transition (see Proposition~\ref{prop:adj}).  

In contrast, for the Laplacian, there is a transition at $p = \log n / n,$ before which point the graph has isolated vertices.  Each isolated vertex contributes a $0$-eigenvalue to the spectra of the Laplacian, but as a consequence of Theorem~\ref{thm:ergap_giant}, the remaining eigenvalues will be $1 + O(1/\sqrt{np})$ as anticipated. There is a second transition at $p = \tfrac 12 \log n/n$ below which there are quadruplets of vertices in the giant component on which the induced graph is a path.  These quadruplets each contribute an eigenvalue near to $\tfrac 12,$ but the remainder of the spectra will again be $1+O(1/\sqrt{np}).$  Continuing in this way, we conjecture that there are a whole family of transitions at $\tfrac 1 k \log n / n$ for any natural number $k,$ where the spectral gap of the giant component is asymptotically the spectral gap of a path on $k$ vertices.  
%We believe the arguments in this paper could easily be generalized to give a more general structure theorem that holds for any $p \sim c \log n / n$ for $c>0.$  For example, for $\tfrac 12 \log n / n > p > \tfrac 13 \log n / n,$ it should be possible to decompose the spectra of the normalized Laplacian as $0$-eigenvalues (one for each isolated vertex),  eigenvalues nearly equal to $\tfrac 12$ (one for each path of length $2$ that appears as an induced subgraph of the graph), and the bulk of the eigenvalues which have order $C/\sqrt{np}.$  Analogous statements should hold for $\tfrac 1 k \log n / n > p > \tfrac 1{k+1} \log n / n$ for any natural $k.$

\subsection{Comparing spectra and the gap theorem proof approach} \label{sec:gap_proof}

While it is relatively straightforward to transfer estimates on the gap of $A$ to the gap of $\EL$ in the $p = \omega(\log n / n)$ regime, Coja-Oghlan~\cite{CojaOghlan} sharpens this analysis to show that there are $c>0$ and $K>0$ so that for $p \geq c \log n/n,$ all but the smallest eigenvalue of $\EL$ are at most $K/\sqrt{np}$ in modulus with high probability.

There are some similarities between our approach and the method of Coja-Oghlan~\cite{CojaOghlan}.  His analysis rests on applying the Kahn-Szemer\'edi machinery to the adjacency matrix of a sufficiently regular subgraph of $G(n,p)$ and then arguing this core of the graph determines the eigenvalues of the Laplacian of the whole graph.  We make a finer analysis of the structure of $G(n,p)$ in the $p = \Theta(\log n / n)$ regime in order to show that in fact the spectra of the adjacency matrix and the spectra of the normalized Laplacian only fail to be comparable when small sparse subgraphs appear.

To bound $\max_{i > 1} \left| 1 - \lambda_i\right|$ it suffices instead to bound the spectrum of what is essentially $I - \EL.$ Given the graph $G$ with vertices $\{1,2,\dots,n\}$ we define the matrix
\[
M_{u,v} =
\begin{cases}
\frac{1}{\sqrt{\deg(u)}\sqrt{\deg(v)}}& \text{if $u$ is adjacent to $v$,}
\\
0 &\text{otherwise.}
\end{cases}
\]
Thus if all degrees are positive we have
\[M = T^{-1/2}AT^{-1/2},\]
and it is easily checked that for any vertex set $W$ of a connected component of $V$, $T^{1/2}\one_{W}$ is an eigenvector with eigenvalue one.

%Like in~\cite{CojaOghlan}, we prove this theorem in part by comparison with the spectrum of the adjacency matrix, which exhibits no transition in the $p = \Omega(\log n/ n)$ regime.  The spectra of the adjacency matrix is handled by the technique of Kahn and Szemer\'erdi, and is a slight improvement over an earlier estimate of Feige and Ofek~\cite{FeigeOfek}.
Set $S=\{x~\vert~x^t\one = 0\}.$  The standard Kahn-Szemer\'erdi machinery applied to the adjacency matrix shows that
\[
|x^tAy| \leq C\sqrt{d} \|x\|\|y\|,
\]
where $d = np,$ for all $x\in S$ and all $y \in \R^n,$ provided $p = \Omega(\log n / n).$

When $p > (1+\epsilon) \log n / n,$ the comparison is relatively straightforward, by virtue of the fact that with high probability all the degrees in the graph are larger than $d/M$ for some sufficiently large $M$.  In particular, this means that $\|T^{-1/2}\| \leq \sqrt{M}/\sqrt{d}.$  One must additionally show that $T^{-1/2}\one$ is nearly parallel to $\one,$ i.e. $T^{-1/2}$ nearly maps the space $S$ to itself.  In sum, these two facts show that for $x\in S,$ $T^{-1/2}x$ is still nearly in $S$ and has norm $\|T^{-1/2}x\| \leq \sqrt{M}\|x\|/\sqrt{d}.$  Thus,
\[
  |x^t M x| = |(T^{-1/2}x) A (T^{-1/2}x)| \approx C\sqrt{d} \|T^{-1/2}x\|^2  \leq C M \|x\|^2/\sqrt{d},
\]
giving the desired result.

Likewise, when $p > \frac{\log n + (\log n)^{1/2+\delta}\log\log n}{n},$ the minimal degree of the graph is still at least $d^{1/2+\delta}$ w.h.p.  In this case, the $T^{-1/2}$ still nearly maps $S$ to $S$, but now $\|T^{-1/2}x\| \leq d^{-1/4-\delta/2}.$  This allows one to show that
\[
\max_{i > 1} \left| 1 - \lambda_i\right| < d^{-\delta},
\]
which is essentially the approach taken by an earlier version of this paper.

To get theorems that hold all the way down to below $p = \log n/n$, where the minimum degree drops to $0$ an additional argument is needed.  This is because it is no longer the case that $\|T^{-1/2}\| = O(1/\sqrt{d}).$ The key structure theorem that allows the comparison to go through is an analysis of the graph structure surrounding low-degree vertices.  Precisely, we show that near the connectivity threshold, there are no edges between low-degree vertices, and low-degree vertices do not even have shared neighbors (see Proposition~\ref{prop:key_structure}).  Thus, they are only connected through the large, high-degree core.  This is enough to ensure that the desired spectral properties persist all the way down to around $p \sim 1/2 \log n / n.$

On the other hand, below $p \sim 1/2 \log n / n,$ low-degree vertices in the giant component begin to connect with high probability.  Indeed, it is possible to show that there are even two degree $2$ vertices that connect to each other and the high-degree core.  This is enough to ensure that $\lambda_2$ of the giant component is at most a little above $\tfrac12$ and $\lambda_n$ is at least $\tfrac32.$

\subsection{Further Discussion}\label{sec:discussion}

For $p$ satisfying $np - \log n \to \infty,$ we have provided a bound on $\lambda(G)$ that is sharp up to a constant multiplicative factor.  For the adjacency matrix in many regimes, much more is known about the behavior of the second largest eigenvalue.

Recall that a Wigner matrix is a symmetric matrix with independent, centered, variance $1$ entries above the diagonal.
From Wigner's celebrated semicircle law, it can be inferred that the largest eigenvalue of such a matrix is around $2\sqrt{n}.$  In fact a much stronger result is known for a large class of Wigner matrices,
%(\cite{Soshnikov})
for which it is seen that
\[
  n^{1/6}(\lambda_1 - 2\sqrt{n}) \weakto X
\]
where $X$ follows the GOE Tracy-Widom law.  When the entry distributions are $\Bernoulli(p)$ -- i.e. when this is the adjacency matrix of an \ER graph -- it was recently shown by Knowles, L. Erd\H{o}s, Yau and Yin~\cite{KEYY} that for $p \gg n^{-1/3},$ the analogous results hold for the second largest eigenvalue.  One of the limits of comparing the spectra of the adjacency matrix and the Laplacian matrix is that such a fine statement about the spectra does not easily transfer.  It is appealing to speculate that at $p \sim \log n/ n,$ the smallest nonzero eigenvalue of the normalized Laplacian is exactly $1-(2-o(1))\sqrt{np},$ consistent with what would be predicted by the semicircle law of the adjacency matrix.

The spectral gap of the normalized Laplacian is strongly related to other probabilistic quantities of the graph, in particular to properties of simple random walk (see~\cite{Chung} for more details) and to the Cheeger constant.  Direct analysis of these quantities is also possible, which then implicitly give bounds on the spectral gap.
Benjamani et.\ al. take a combinatorial approach and study the Cheeger constant (also called isoperimetric constant, or conductance) throughout the evolution of the random graph process \cite{cheegerconstant}.  Likewise Fountoulakis and Reed study the mixing time of simple random walk on the giant component through the conductance~\cite{FR} in the strictly supercritical regime $\tfrac{1+\epsilon}{n} < p < \tfrac{\sqrt{\log n}}{n}.$
Ding et.\ al.\ studied probabilistic aspects of the graph including the mixing time of simple random walk on the giant component as the graph emerges from the critical window \cite{peresrandomwalk}.  All these works show that the giant component can be partitioned into a well connected expanding core together with small (logarithmic size) graphs attached to the core.  We also employ a version of this decomposition to analyze the spectral properties of the graph.

\section{Random topology} \label{sec:rt}

In~\cite{LM06},
Linial and Meshulam introduce an analogous measure $Y_2(n,p)$ to the binomial random graph for random $2$-dimensional simplicial complexes.
This is a probability distribution over all simplicial complexes with vertex set $[n]=\{1,2,\dots,n\}$ with complete $1$-skeleton (i.e.\ with all possible $n \choose 2$ edges). Each of the ${n \choose 3}$ possible $2$-dimensional faces are included independently with probability $p$. We use the notation $Y \sim Y_2(n,p)$ to indicate a complex drawn from this distribution.
Meshulam and Wallach~\cite{MW09} extend this definition to a $d$-dimensional complex $Y_d(n,p)$, formed by taking the complete $(d-1)$-skeleton of the $n$-vertex simplex, and including $d$-dimensional faces independently with probability $p.$

The distributions can be made into stochastic growth processes in a natural way.
Let $Y_2(n,m)$ be the random $2$-complex that has the uniform distribution over all simplicial complexes with $n$ vertices, $n \choose 2$ edges, and exactly $m$ two-dimensional faces. In the random complex process $\{ Y_2(n,m) \}$, faces are added one at a time, uniformly randomly from all faces which have not already been chosen.  In the same way, we can define the process $\{Y_d(n,m)\}$ by including $d$-faces one at a time.

We also define a time-changed version of this process $\LMP,$ more suitable to working with the binomial complex.  Instead of including the faces one at a time, create independent $\Exponential(1)$ clocks for every $d$-face.  When one of the clocks rings, include the corresponding face.  If we let $p(t) = 1 - e^{-t},$ then $\LMP$ has the distribution $Y_d(n,p(t)).$

\subsection{Cohomology vanishing}

The foundational work on the Linial-Meshulam complexes is a cohomological analogue of the~\ER connectivity theorem.
%The following theorem is proven by a delicate cocycle counting argument.
\begin{LMW}
\label{thm:LMW}
Let ${\bf k}$ be any finite field, $d \ge 2$ fixed, $f(n) \to \infty$ be any slowly growing function, and $Y \sim Y_d(n,p)$.
If
$$ p \ge \frac{d\log n + f(n)}{n},$$
then w.h.p.\ $H^{d-1}(Y, {\bf k}) = 0$, and
if
$$ p \le \frac{d\log n-f(n)}{n},$$
then w.h.p.\ $H^{d-1}(Y, {\bf k}) \neq 0$.
\end{LMW}
For the case that $d=2$ and ${\bf k}={\Z}_2,$ this is due to Linial and Meshulam~\cite{LM06}, while for the version stated, this is due to Meshulam and Wallach~\cite{MW09}.  By the universal coefficient theorem, these results imply the corresponding theorem for the cohomology with $\Q$ coefficients.  For $\Z$ coefficients, it is shown by the authors in~\cite{HKP} that for $p \geq 80d\log n / n$, $H^{d-1}(Y, \Z) = 0$ by other techniques.  For $d=2,$ work of \cite{LuczakPeled} establishes $2\log n/n$ as the sharp threshold for vanishing $\Z$ homology.

The threshold $p \sim d\log n / n$ is also the threshold for the existence of isolated $(d-1)$-faces in the complex, i.e. faces that are not included in any $d$-face.  Indeed, the presence of isolated faces is precisely the reason that the cohomology is nonzero below this threshold.  In fact, a finer statement can be made about the number of isolated $(d-1)$-faces.
\begin{lemma}
\label{lem:poisson_faces}
Let $I$ denote the number of isolated $(d-1)$ faces in $Y_d(n,p).$  Suppose that for fixed $c,$
\[
p = \frac{d \log n + c+o(1)}{n}.
\]
Then $I$ converges in law to $\Poisson( e^{-c}/d!).$
\end{lemma}
The proof of this lemma is standard and can be proved in the same manner as the Poisson convergence of the number of isolated vertices in $G(n,p).$  See Proposition 4.13 of~\cite{Ross}.

Using spectral techniques, we give a new proof of the Linial--Meshulam--Wallach theorem, although only with $\Q$ or $\R$ coefficients.  However, for $\Q$ coefficients, we also sharpen the theorem by proving a process version.  More strikingly, this theorem shows that long before the last isolated $(d-1)$-faces disappear, the only obstruction to vanishing cohomology are those isolated $(d-1)$-faces.  Its proof follows almost immediately from spectral arguments and Garland's method (see Section~\ref{sec:cohom}).

\begin{theorem} \label{thm:Hhit}
Consider the random complex process $\{ \LMP \}$.  Let $I_t$ denote the number of isolated $(d-1)$-faces in the complex at time $t.$  Fix any $\delta >0$ and define $t_0$ so $p(t_0) = (d-1+\delta)\log n / n.$  Then w.h.p.\ for all time $t \geq t_0,$
\[
H_{d-1}(\LMP, \Q) \cong \Q^{I_t}.
\]
\end{theorem}

As w.h.p. $I_{t_0} > 0$ we immediately get the following hitting time corollary.

\begin{corollary} \label{cor:Hhit} Consider the random complex process $\{ Y_d(n,m) \}$. Let
$$M_1 = \min \{ m \mid Y_d(n,m) \mbox{ has no isolated } (d-1)-\mbox{dimensional faces}\},$$
and let
$$M_2 = \min \{ m \mid H^{d-1}(Y_d(n,m), \Q) = 0 \}.$$
Then w.h.p.\ $M_1 = M_2$.
\end{corollary}

Further, it is standard to show at this point that the Betti numbers are asymptotically Poisson.
\begin{corollary} \label{cor:HPoisson} Suppose that for fixed $c,$
\[
p = \frac{d \log n + c+o(1)}{n}.
\]
Then $b_{d-1}(Y_d(n,p))$ converges in law to $\Poisson( e^{-c}/d!).$
\end{corollary}
Note that this follows immediately from Lemma~\ref{lem:poisson_faces} and Theorem~\ref{thm:Hhit}.

\subsection{The fundamental group}
\label{sec:pi1}

For the $2$-dimensional complex, a fair bit is known about the fundamental group $\pi_1(Y).$ Babson and the first two authors find the threshold for the fundamental group to be trivial~\cite{bhk11}.
\begin{theorem} [Babson--Hoffman--Kahle]
If
$p = n^{-\alpha}$ where $\alpha < 1/2$ then w.h.p. $\pi_1(Y)$ is a nontrivial word hyperbolic group.  If $p\geq n^{-1/2}\log(n)$ then $\pi_1(Y)$ is trivial.
\end{theorem}
Cohen et al. \cite{CCFK12} show that if
$p = o ( 1/n),$
then w.h.p.\ $\pi_1(Y)$ is free.  Finally, Costa and Farber describe the cohomological dimension $\cd \pi_1(Y)$ in various regimes \cite{CF12,CF13}.
\begin{theorem} [Costa--Farber] Let $Y \sim Y_2(n,p)$, and set $p = n^{-\alpha}$.
\begin{enumerate}
\item If $\alpha > 1$ then w.h.p.\ $\cd \pi_1(Y) = 1$,
\item if $1 > \alpha > 3/5$ then w.h.p.\ $\cd  \pi_1(Y)=2$, and
\item if $3/5 > \alpha > 1/2$ then w.h.p.\ $\cd \pi_1(Y) = \infty$.
\end{enumerate}
\end{theorem}

For the $2$-dimensional complex, we combine the new spectral results with Garland's method to show a threshold theorem for $\pi_1(Y)$ to have property~(T).
A group $G$ is said to have property~(T) if every unitary action of $G$ on a Hilbert space that has almost invariant vectors also has a nonzero invariant vector.  The first explicit examples of expanders, due to Margulis, were constructed using Cayley graphs on quotients of (T) groups such as $SL(3,\Z)$ \cite{Margulis}.  Conversely, expansion properties of some graphs associated to the generating set of a group can imply property~(T) (see~\cite{Zuk}).

Property~(T) has found use in many different areas of mathematics.
For example, groups with property~(T) lead to good mixing properties in ergodic theory --- a process which mixes slowly must leave some subsets almost invariant. In particular, if a group $\Gamma$ has property~(T), then every ergodic $\Gamma$ system is also strongly ergodic \cite{glasner}.  See the monograph \cite{Bekka} for a comprehensive overview of property~(T).

We recall for convenience the statement of Theorem \ref{thm:structure}:
\begin{theorem*}
Suppose $\delta > 0$ is fixed,
$$ p \ge \frac{ ( 1 + \delta) \log n}{n},$$
and $Y \sim Y_2(n,p)$.
Then w.h.p.\ $\pi_1(Y)$ is isomorphic to the free product of a (T) group $G$, and a free group $F$, where the free group $F$ has one generator for every isolated edge in $Y$.
\end{theorem*}

Theorem \ref{thm:structure} might be viewed as a group-theoretic analogue of the fact that for $p \ge (1/2 + \delta) \log n / n$, the random graph $G \sim G(n,p)$ is w.h.p.\ a giant component, which is an expander, and isolated vertices.

We anticipate that the true threshold for $\pi_1(Y)$ being the free product of a free group and a nontrivial (T) group is much lower, and that it occurs in the range $p = \Theta(1/n).$  The significance of the threshold $\log n /n$ is that this is the threshold at which the free group is generated by isolated edges. 

For example, if $p = \delta \log n / n$ with $0 < \delta < 1$ fixed, then w.h.p.\ there exists a triangle $abc$ in $Y_2(n,p)$ such that edges $ab$ and $ac$ are not contained in any other triangle. In other words, the edge $bc$ is a connected component in the link of vertex $a$. In this case, the edge $ab$ and triangle $abc$ can be collapsed by an elementary collapse. This is a homotopy equivalence---after the collapse, the edge $ac$ is a generator of a free $\Z$ factor in $\pi(Y)$, but before the collapse there is no isolated edge generating this element of the group. 

On the other hand, this is also the point our argument in Section 8 ceases to apply. To apply \.Zuk's criterion in Section 8, we first delete all isolated edges and the resulting complex has connected vertex links with good expansion properties. In the case above, the vertex link $a$ is not connected, even after such deletions.

We have the following corollary of Theorem \ref{thm:structure}, which shows that the threshold for property~(T) is the same as the Linial--Meshulam theorem for vanishing of $\Z / 2$-homology.

\begin{corollary} \label{cor:Twin}
Let $\omega \to \infty$ as $n \to \infty$, and $Y \sim Y_2(n,p)$.
If $$p \ge \frac{2 \log{n} + \omega}{n}$$  then $\Prob[\pi_1(Y) \mbox{ has property~(T)}] \to 1$.  \label{T}
\end{corollary}

We also describe a process version of this structure theorem that holds below the connectivity threshold. 

%Let $F_k$ denote the collection of all $k$-faces of the $n$-simplex, and let $\{T_\sigma, \sigma \in F_k\}$ be an i.i.d.\ family of $\Exponential(1)$ variables.  Define $\{\LMP, t \geq 0\}$ to be the continuous time Markov process
%where $\LMP[0]$ is the complete $(k-1)$-skeleton of the $n$-simplex and its $k$-faces are given by
%\[
%F_k(\LMP) = \{ \sigma \in F_k~:~T_\sigma \leq t\}.
%\]
%Thus $\LMP$ is a complex whose $k$-faces have been born up to time $t,$ and $\LMP[\infty]$ is the complete $k$-skeleton of the $n$-simplex.

%We also define a time-changed version of this process $\LMP = Y_d(n,N_t)$ where $N_t$ is a Poisson process with rate ${n\choose d+1}$.  If we let $p(t) = 1 - e^{-t},$ then $\LMP$ has the distribution $Y(n,p(t)).$

\begin{theorem} \label{thm:Thit}
Consider the random complex process $\{ \LMP \}$.  Let $\tilde F_t$ be a free group with the number of generators equal to the number of isolated edges in the complex $\LMP$.  Fix any $\delta >0$ and define $t_0$ so $p(t_0) = (1+\delta)\log n / n.$  Then w.h.p.\ for all $t \geq t_0,$
\[
\pi_1(Y_2(n,p(t))) \cong G_{t} * \tilde F_{t}
\]
where $G_t$ has property~(T).
\end{theorem}

Note that Theorem~\ref{thm:structure} follows immediately from this.
As the number of isolated edges at time $t_0$ is positive w.h.p, we get the following hitting time corollary.

\begin{corollary} \label{cor:Thit}
Consider the random complex process $\{ Y_2(n,m) \}$. Let
$$M_1 = \min \{ m \mid Y_2(n,m) \mbox{ has no isolated edges} \},$$
and let
$$M_2 = \min \{ m \mid \pi_1( Y_2(n,m) ) \mbox{ is (T)} \}.$$
Then w.h.p.\ $M_1 = M_2$.
\end{corollary}

\begin{remark}
%{\bf Kazhdan pair}
We can additionally give an explicit Kazhdan pair for the (T) group.  Setting $S$ to be the canonical generating set based at vertex $1,$ i.e. all loops cycles of the form $1 \to x \to y \to 1$ for distinct vertices $x$ and $y,$ then $(S,\sqrt{2}(1-o(1)))$ is a Kazhdan pair (see Remark 5.5.3 of~\cite{Bekka}).
\end{remark}

\section{Spectral estimates} \label{sec:est}

In this section we give some conditions on an arbitrary graph $G$
on $n$ vertices which facilitate a large spectral gap.   Fix positive constants $C_1, C_2, C_3$ and $M$.
%% NOTE: we could add this line back in, but thinking of an arbitrary number as 'large' seems pretty vacuous to me. --EP
%(Although these constants are arbitrary it makes sense to think of $C_1,C_2,C_3$ and $M$ as large positive numbers.)
In this section $d$ can be any function of $n$ with $d=d(n) \geq 1,$ and this is always satisfied by $d=(n-1)p,$ the convention taken in other sections.
%but for consistancy with other sections, we will assume $d=(n-1)p.$
%and $d$ as roughly the average degree of $G$.)
%Let $d$ be the average degree of $G$.

Recall that $T$ is the diagonal matrix of degrees.
Let $W$ denote the set of vertices $x$ for which $\deg x > 0$
%and let $\fuzz$ is as in Proposition~\ref{prop:key_structure}.
and $I$ be the number of isolated vertices in the graph.
For any set of vertices $S,$ let $\one_{S}$ denote the vector that is one in every coordinate corresponding to $S$ and $0$ elsewhere.
Let $0= \lambda_1 \leq \lambda_2 \leq \cdots \leq \lambda_{n}$ be the eigenvalues of the normalized Laplacian $\EL[G],$ so that $\lambda_1 = \lambda_2 = \cdots = \lambda_{I+1} = 0.$
 We also define a set of vertices of small degree. %For any $M \geq 1,$ d
 Let
\begin{equation} \label{the fuzz}
\fuzz = \{
v \in V ~:~
\deg(v) \leq d/M
\}.
\end{equation}

%Suppose we are given a fixed graph $G$ on $n$ vertices.
%Suppose that there exists constants $C_1, C_2, C_3, M$ and $d$ such that the following conditions are satisfied by the graph.

We now define four conditions that will ensure a spectral gap.
\begin{enumerate}
\item {\bf Bounded degree condition (b.d.c)} Every vertex has degree at most $C_1 d.$
\item {\bf Adjacency matrix} %There is a constant $C_2$ so that
\[
\sup_{ \substack{ \|x\| = 1, x^t \one = 0 \\ \|y\| = 1 } } | x^tAy| \leq C_2\sqrt{d}.
\]
\item {\bf Fuzz} There are no edges between vertices of $\fuzz,$ $|\fuzz| \leq \frac{n}{2}$ and \[ \max_{u \in \fuzz^c} \dset{u}{\fuzz} \leq 1,\]
  where $e(U,V)$ denotes the number of edges between sets of vertices $U$ and $V.$
\item {\bf Parallel eigenspaces}  %There is a constant $C_3$ so that
\[
\sup_{ \substack{ \|x\| = 1, \\ x^t T^{1/2}\one_{W} = 0 } } | x^t T^{-1/2} \one_{\fuzz^c}| \leq C_3\frac{\sqrt{n}}{d}.
\]
\end{enumerate}
The final condition states that a vector $x$ that is orthogonal to the kernel of $L$ will not have such a large component in the direction of the principal eigenvector of $T^{-1/2}A$.  The vector $\one_{\fuzz^c}$ can be considered as a good approximation to this principal right eigenvector.  Otherwise said, the $0$--eigenspace of $L$ and the principal right eigenspace of $T^{-1/2}A$ are nearly parallel.

With these definitions we can now state our main result on spectral gaps.
\begin{lemma}
\label{lem:deterministic}
Let $G$ be a graph on $n$ vertices and let $C_1,C_2,C_3$ and $M$ be constants.
If $G$ satisfies the four conditions above
then there is a constant $C=C(C_1,C_2,C_3,M)$ so that
\[
\max_{i > I + 1} \left| 1 - \lambda_i\right| < \frac{C}{\sqrt{d}}.
\]
\end{lemma}
\begin{proof}
%[Proof of Theorem~\ref{thm:ergap_giant}]
Let $W$ be the set of vertices $x$ for which $\deg x > 0.$  By the spectral theorem, $\EL$ admits a basis of orthogonal eigenvectors.  Let $v$ be a normalized eigenvector of $\EL$ corresponding to an eigenvalue $\lambda_i$ with $i > I+1.$  Setting $l_1,l_2,\ldots, l_I$ to be the isolated vertices, a basis for the kernel of $\EL$ is given by $\{ T^{1/2}\one, \delta_{l_1},\delta_{l_2}, \ldots, \delta_{l_I}\},$ where $\delta_a$ is $1$ in the $a^{th}$ coordinate and $0$ elsewhere.  As $v$ is orthogonal to all of these, it is orthogonal to $T^{1/2}\one_{W}.$  Hence,
\[
\left| 1 - \lambda_i\right|
=
\left| v^t T^{-1/2} A T^{-1/2} v\right|
 \leq
\sup_{\substack{ \|x\| = 1, \\ x^tT^{1/2}\one_{W}=0}}
\left|
x^t T^{-1/2} A T^{-1/2} x
\right|.
\]
As this holds for all such $i > I+1,$ it suffices to bound the right hand side.

Orthogonally decompose $T^{-1/2}x = u + v,$ where $u$ is supported on vertices in $\fuzz^c$ and $v$ is supported on vertices in $\fuzz.$  Further decompose $u = u_0 + u_1$ by letting $u_1$ be the projection of $u$ along $\one_{\fuzz^c}.$
%and by the parallel eigenspaces condition {\bf ???}
%\begin{equation}  \label{for Tacoma}
%\|u_1\| \leq C_3/d.
%\end{equation}
Expanding the quadratic form, we may write
\begin{equation} \label{decomposition}
\left|
x^t T^{-1/2} A T^{-1/2} x
\right|
\leq 2|u_0^tAu| + |u_1^tAu_1| + |v^tAv| + 2|v^tAu|.
\end{equation}
Each of these terms will be seen to have the right order bound, completing the proof.

As $u_0 \perp \one_{\fuzz^c}$ and is supported only on $\fuzz^c,$ we have that $u_0 \perp \one.$
By the definitions of $\fuzz$ and $x,$ we have that
%{\bf (more explanation?)}
%\begin{equation} \label{T it up}
\[
\|u_0\|^2 \leq \|u\|^2
= \sum_{i \in \fuzz^c} \frac{|x_i|^2}{\deg i}
\leq \frac{M}{d}.
\]
%\end{equation}
Hence by the adjacency matrix condition and the above equation %and (\ref{T it up}),
we have that
\begin{equation} \label{partone}
|u_0^tAu| \leq C_2\sqrt{d} \|u_0\| \|u\| \leq \frac{C_2M}{\sqrt{d}}.
\end{equation}

As $u_1$ is the projection of $u$ along $\one_{\fuzz^c},$ we have
\[
u_1
= (u^t \one_{\fuzz^c}) \frac{ \one_{\fuzz^c}}{|\fuzz^c|}
= (x^t T^{-1/2} \one_{\fuzz^c}) \frac{ \one_{\fuzz^c}}{|\fuzz^c|}.
\]
Because $|\fuzz^c| \geq \frac{n}{2},$ the parallel eigenspaces condition implies that we have
\(
\|u_1\| \leq \frac{\sqrt{2}C_3}{d}.
\)
The norm of $A$ is at most the maximum degree of the graph, and by the bounded degree condition this is at most $C_1d.$  Hence, we get that
\begin{equation} \label{parttwo}
|u_1^tAu_1| \leq \frac{2C_1C_3^2}{d}
\end{equation}

For the third term, we note that by the $\fuzz$ condition there are no edges between vertices of $\fuzz,$ and hence
\begin{equation} \label{partthree}
v^tAv = 0.
\end{equation}

Finally, we may expand $v^tAu$ as
\[
v^tAu =
\sum_{i \in \fuzz} \frac{x_i}{\sqrt{\deg i}}
\sum_{\substack{ j\in \fuzz^c,\\ j \sim i}} u_j.
\]
By Cauchy-Schwarz, this is bounded by
\[
|v^tAu|^2 \leq
\sum_{i \in \fuzz} \frac{1}{{\deg i}}
\biggl(\sum_{\substack{ j\in \fuzz^c,\\ j \sim i}} u_j\biggr)^2
\leq
\sum_{i \in \fuzz}
\sum_{\substack{ j\in \fuzz^c,\\ j \sim i}}
\left(u_j\right)^2.
\]
Now each $j \in \fuzz^c$ has at most one neighbor in $\fuzz$, and hence we have
\begin{equation} \label{partfour}
\left|v^tAu\right| \leq \|u\| = \frac{\sqrt{M}}{\sqrt{d}}.
\end{equation}

Plugging (\ref{partone}), (\ref{parttwo}), (\ref{partthree}) and (\ref{partfour}) into (\ref{decomposition}) completes the proof. \end{proof}

In the remainder of this section we prove
a condition on a graph that will imply an upper bound on the spectral gap. This lemma shows that our previous argument breaks down when the set $\fuzz$ fails to be isolated.
%the remainder of Theorem~\ref{thm:ergap_threshold}.  The upper half, where $p \geq (\tfrac12 + \delta) \log n / n,$ follows directly from Theorem~\ref{thm:ergap_giant}, and so it remains to prove the lower half.  Our main tool is the following lemma.
\begin{lemma}
\label{lem:deterministic2}
Suppose that $H$ is a connected graph and that there are vertices $u,v,w,x$ for which the induced graph on $u,v,w,x$ is a path with endpoints $u$ and $x$.  Suppose further that $\deg v = \deg w = 2$  and $\deg u, \deg x \geq m.$ %and $\deg x \geq m.$
Let $0= \lambda_1 \leq \lambda_2 \leq \cdots \leq \lambda_{|H|}$ be the eigenvalues of the normalized Laplacian $\EL[H],$ then
\[
\lambda_{|H|} \geq \tfrac{3}{2}
\]
and
\[
\lambda_{2} \leq \tfrac{1}{2} + O(1/\sqrt{m})
\]
\end{lemma}
\begin{proof}
For each case, we construct an appropriate approximate eigenvector.  For the first, consider the vector $f$ with $f(v) = 1,$ $f(w)=-1$ and $f(y)=0$ for all other $y$.  This vector is orthogonal to $T^{1/2}\one$, the first eigenvector of $\EL.$
Now $T^{-1/2}f$ is just $f/\sqrt{2}$ while $f^tAf = -2.$  Thus,
\[
\frac{f^tT^{-1/2}AT^{-1/2}f}{\|f\|^2} = -\frac{1}{2},
\]
and so $\lambda_{|S|} \geq 1 - \tfrac{-1}{2} = \tfrac{3}{2}.$

For the lower bound let $f$ be given by $f(v) = f(w) = 1/\sqrt{2}$ while $f(x) = -1/\sqrt{\deg x}$ and $f(u) = -1/\sqrt{\deg u}.$  Then we have $f \perp T^{1/2}\one.$  By direct computation,
\[
{f^tT^{-1/2}AT^{-1/2}f} = \frac{1}{2} - \frac{1}{\deg x} - \frac{1}{\deg u},
\]
while
\[
\|f\|^2 \leq 1 + \frac{1}{\deg x} + \frac{1}{\deg u}.
\]
Thus, combining everything, we have that
\[
\lambda_{2} \leq 1- \frac{\tfrac12 - \tfrac{2}{m}}{\sqrt{1+\frac{2}{m}}}
= \tfrac{1}{2} + O(1/\sqrt{m}).
\]
\end{proof}

\section{Probability bounds} \label{sec:pbnds}

In this section we show various estimates on $G(n,p),$ which when combined with the deterministic lemmas on spectral gaps from the previous section, will complete the proofs of Theorems~\ref{thm:ergap_giant} and \ref{thm:ergap_giant2}. In this section we again use that $d=(n-1)p$ is the expected degree of a vertex. 
%We will define a series of good events that depend on constants, which we will determine later.  These constants may depend on $\delta$ and $\epsilon,$ and they may additionally depend on previous constants.  However,

\begin{lemma}
\label{lem:regularity}
For each $\delta >0$ and $m \geq 0,$ there is a constant $C=C(\delta, m)$ so that the following conditions hold
with probability at least $1-C\exp(-md)$ and $\rsp$ respectively in $G(n,p)$ with $p \geq \delta \log n / n$.
\begin{enumerate}
\item {\bf Bounded degree condition (b.d.c)} Every vertex has degree at most $C d.$
\item {\bf Discrepancy} For every pair of vertex sets $A$ and $B$, letting $e(A,B)$ denote the number of edges between the sets and $\mu(A,B) = \tfrac{|A||B|d}{n}$, one of
\begin{enumerate}
\item
$
\tfrac{e(A,B)}{\mu(A,B)} \leq C
$
\item
$
e(A,B)\log \tfrac{e(A,B)}{\mu(A,B)} \leq C( |A| \vee |B|)\log \tfrac{n}{|A| \vee |B|}
$
\item
$
|A| \leq d^{1/4}/100, |B| \leq d^{1/4}/100
$
\end{enumerate}
occurs.
\end{enumerate}
\end{lemma}

Both of these bounds are consequences of tail bounds of binomial variables, and they are relatively standard in the literature (see, e.g. \cite{FKSz},\cite{FeigeOfek},\cite{CojaOghlan}).  This one differs in that we look for more control over the order of decay of the failure probability.

\begin{proposition}
\label{prop:adj}
For each $\delta >0$ and $m \geq 0,$ there is a constant $C=C(\delta, m)$ sufficiently large so that if
\(p \geq \delta \log n / n\)
then
\[
\sup_{ \substack{ \|x\| = 1, x^t \one = 0 \\ \|y\| = 1 } } | x^tAy| \leq C\sqrt{d}
\]
with probability at least $\rsp - C\exp(-md).$
\end{proposition}

This follows from the standard Kahn-Szemer\'erdi argument, and it is essentially proven in both Feige and Ofek~\cite{FeigeOfek} and the original Friedman, Kahn and Szemer\'erdi paper~\cite{FKSz}.  This version has a sharper estimate on the failure probability than~\cite{FeigeOfek}, which in turn follows from Lemma~\ref{lem:regularity}.  We will delay the proof of both this and the previous lemma to Section~\ref{sec:KSz}.

Additionally, the bounded degree condition is needed to make estimates about low degree vertices.  Recall the definition of $\fuzz$ from (\ref{the fuzz}).
We show that this set is both small and structurally very simple for sufficiently large $M.$

\begin{proposition}
\label{prop:key_structure}
For each $\delta >0$ and each $\epsilon >0,$ 
 there is an $M=M(\delta,\epsilon) > 1$
 such that for \(p \geq  (\tfrac 12 + \delta)\log n / n\), $G(n,p)$ satisfies:
\begin{enumerate}
\item $|\fuzz| < n/(100d)$
\item $\fuzz$ is an independent set,
\item and \( \max_{u \in \fuzz^c} \dset{u}{\fuzz} \leq 1\)
\end{enumerate}
with probability at least $1 - Cn\exp(-(2-\epsilon)d)- C\exp(-cn)$ for some absolute constant $c>0.$
\end{proposition}

\begin{proof}

\noindent{\emph{ (i)}} We start by estimating the size of $\fuzz,$ which we do by a simple union bound.  Namely by symmetry we have
\begin{align*}
\Pr\left[
|\fuzz| \geq k
\right]
\leq
{n \choose k}
\Pr\left[
\deg u_i \leq d/M, 1 \leq i \leq k
\right].
\end{align*}
Let $S$ be the set of vertices $u_{k+1},\ldots,u_n,$ then we have
\[
\Pr \left[
\deg u_i \leq d/M, 1 \leq i \leq k
\right]
\leq
\Pr \left[
\dset{u_i}{S} \leq d/M, 1 \leq i \leq k
\right],
\]
which are now independent $\Binomial(n-k,p)$ variables.  Applying Lemma~\ref{magic}, we get
\begin{align*}
\log \Pr\left[
|\fuzz| \geq k
\right]
\leq k\left[
(1+ \log\frac{n}{k})-(d - kp) + \frac{d}{M}( 1 + \log(M))
\right].
\end{align*}
Setting $k = [n/(100d)],$ we may make $M$ sufficiently large that
\[
(1+ \log\frac{n}{k})-(d - kp) + \frac{d}{M}( 1 + \log(M)) \leq -\frac{d}{2}
\]
for all $n \geq n_0(\delta).$
Hence we have that $|\fuzz| < n/(100d)$ with probability at least $1 - O\exp(-cn)$ for some absolute constant $c>0.$

\noindent{\emph{ (ii)}} We begin by bounding the probability that there is an edge between any two vertices of $\fuzz.$  Note that we may assume that $d < n/100,$ lest $\fuzz = \emptyset$ by the previous bound.

From the union bound and symmetry, we have that
\[
\Pr \left[
\fuzz \text{ is not an independent set}
\right] \leq
n^2 \Pr \left[
v \in S, w \in S, \ebetween{v}{w}
\right].
\]
Thus it suffices to compute this probability, which we do by conditioning $\deg v = d_1$ and $\deg w = d_2.$  Note that the law of the neighborhood $N$ of $\{v,w\}$ under this conditioning is \emph{not} uniform over all such neighborhoods.  For a possible neighborhood $H$ of $\{v,w\},$ let $E(H)$ denote the number of edges in this neighborhood.  Then we have that
\[
\Pr \left[
N = H \middle\vert \deg v = d_1, \deg w = d_2
\right] = \frac{1}{Z} \left(\frac{p}{1-p}\right)^{E(H)},
\]
for a suitable normalization constant $Z.$

Thus, we have that
\begin{align*}
\Pr \left[
\ebetween{v}{w} \middle\vert \deg v = d_1, \deg w = d_2
\right]
&\leq \frac{
\Pr \left[
\ebetween{v}{w} \middle\vert \deg v = d_1, \deg w = d_2
\right]
}{
\Pr \left[
\noebetween{v}{w} \middle\vert \deg v = d_1, \deg w = d_2
\right]
} \\
&=
\frac{1-p}{p}
\frac{
{ n - 2  \choose d_1 - 1}
{ n - 2  \choose d_2 - 1}
}
{
{n - 2 \choose d_1}
{n - 2 \choose d_2}
}.
\end{align*}
As we consider only $d_1$ and $d_2$ that are less than $d / M,$ and as $d < n/100,$ we may bound this as $Cd/n$ for some absolute constant $C.$  It remains to estimate the probability that both $v$ and $w$ are in $\fuzz.$
Hence we have
\[
\Pr\left[
\deg v \leq d/M, \deg w \leq d/M
\right]
\leq
\Pr\left[
X \leq d/M
\right]^2,
\]
where $X \sim \Binom(n-2,p).$  Applying Lemma~\ref{magic}, we have that
\begin{equation}
\label{eq:cofuzz}
\Pr\left[
\deg v \leq d/M, \deg w \leq d/M
\right]
\leq
\exp\left[
-2d + \frac{2d}{M}( 1+ \log M + O(1) )
\right]
\end{equation}
Thus by adjusting $M$ to be sufficiently large, we have
\[
\Pr \left[
\fuzz \text{ is not an independent set}
\right] = O(nd\exp(-(2-\epsilon/2)d)) = O(n\exp(-(2-\epsilon)d)).
\]

\noindent{\emph{ (iii)}}  This follows in much the same way as the proof of (ii).  Here though, we require that the degrees of $\fuzz^c$ are not too large.  By Lemma~\ref{lem:regularity}, these degrees can be bounded by some $Cd$ with probability at least $1-O(\exp(-2d)),$ and so it suffices to assume it.
From the union bound and symmetry, we have that
\begin{multline*}
\Pr \left[
\exists u \in \fuzz^c~:~\dset{u}{\fuzz} \geq 2
~\cap~\bdc
\right] \\
\leq
n^3 \Pr \left[
u \in \fuzz^c, v \in \fuzz, w \in \fuzz, \ebetween{u}{v}, \ebetween{u}{w}~\cap~\bdc
\right].
\end{multline*}
Again we condition on the degrees $\deg u = d_1, \deg v = d_2,$ and $\deg w = d_3,$ and bound
\begin{align*}
\Pr \left[
\ebetween{u}{v}, \ebetween{u}{w}
\middle\vert \deg u = d_1, \deg v = d_2, \deg w = d_3
\right] \hspace{-2in}&\hspace{2in} \\
&\leq \frac{
\Pr \left[
\ebetween{u}{v}, \ebetween{u}{w}
\middle\vert \deg u = d_1, \deg v = d_2, \deg w = d_3
\right]
}{
\Pr \left[
\noebetween{u}{v}, \noebetween{u}{w}, \noebetween{v}{w}
\middle\vert \deg u = d_1, \deg v = d_2, \deg w = d_3
\right]
} \\
&=
\left(\frac{1-p}{p}\right)^2
\frac{
{ n - 3  \choose d_1 - 2}
{ n - 3  \choose d_2 - 1}
{ n - 3  \choose d_3 - 1}
+
\frac{p}{1-p}
{ n - 3  \choose d_1 - 2}
{ n - 3  \choose d_2 - 2}
{ n - 3  \choose d_3 - 2}
}
{
{n - 3 \choose d_1}
{n - 3 \choose d_2}
{n - 3 \choose d_3}
}.
\end{align*}
As before, we have $d_1$ and $d_2$ are less than $d / M.$  As we also require the $\bdc$ to hold, we may take $d_1 \leq \csubone d$ and as $d < n/100,$ we may bound this as $C \csubone d^2/n^2$ for some absolute constant $C.$

From~\eqref{eq:cofuzz}, we have that
\[
\Pr\left[
v \in \fuzz, w \in \fuzz
\right] = O(\exp(-(2-\epsilon/2)d)),
\]
and so we conclude that
\[
\Pr \left[
\max_{u \in \fuzz^c} \dset{u}{\fuzz} > 1
\right] = O(n\exp(-(2-\epsilon)d)).
\]

\end{proof}

%\begin{lemma}
%\label{lem:T_distortion}
%For each $\delta >0$ and $m \geq 0,$ there is a constant $C=C(\delta, m)$ sufficiently large so that if
%\(p \geq \delta \log n / n\)
%then
%\[
%\sum_{v \in V} \left(
%\frac{1}{\sqrt{\deg v}} - \frac{1}{\sqrt{d}}
%\right)^2 \one[\deg v > 0]
%\leq \frac{Cn}{d^2}
%\]
%with probability at least $1-O(\exp(-md)).$
%\end{lemma}
%\begin{proof}

Our next lemma shows that the variance of the degree distribution is not too much larger than its expectation.

\begin{lemma}
\label{lem:T_distortion}
For each fixed $\delta >0$ and $m \geq 0,$ there is a constant $C=C(\delta, m)$ sufficiently large so that if
\(p \geq \delta \log n / n\)
then
\[
\sum_{v \in V} \left(
\deg v - d
\right)^2
\leq Cnd.
\]
with probability at least $1-C\exp(-md).$
\end{lemma}
\begin{proof}
Note that this sum is the square Euclidean norm of the vector $(A-dI)\one.$  Further, it is possible to write the norm as
\[
\| (A-dI)\one \| = \sup_{\|x\|=1} |x^t(A-dI)\one|.
\]
For any fixed vector $x$, we orthogonally decompose it as $x = v + c\one,$ where $|c| \leq 1/\sqrt{n}.$  We have that $v^t(A-dI)\one = v^tA\one,$ and so by Proposition~\ref{prop:adj}, for any $m$ there is a constant $C$ so that
\[
\sup_{\substack{\|v\|=1 \\ v^t\one = 0}} |v^tA\one| \leq C\sqrt{nd}
\]
with probability at least $1-O(\exp(-md)).$  It remains to bound $\one^t(A-dI)\one,$ which is
\[
\one^t(A-dI)\one = \left(\sum_{v \in V} \deg v\right) - nd.
\]
Note that $\sum_{v \in V} \deg v \sim 2 \Binom( {n\choose 2}, p),$ and so by standard Chernoff bounds, we have that
\[
\Pr \left[
\left|\one^t(A-dI)\one \right| \geq t
\right] \leq C\exp( - \frac{t^2}{Cnd} )
\]
for some absolute constant $C$ and all $t \leq nd.$  By taking $t = mn\sqrt{d},$ we have that $\left|\one^t(A-dI)\one \right| \leq mn\sqrt{d}$ with probability at least $1-O(\exp(-mn))$ for sufficiently large $n.$  Recalling that $|c| \leq 1/\sqrt{n},$ we have that
\[
\left|c\one^t(A-dI)\one\right| = O(\sqrt{nd}).
\]
which completes the proof.
\end{proof}
Using the previous lemma, we show that $T^{-1/2}$ tends to map the orthogonal complement of the first eigenvector of $M$ to the approximate orthogonal complement of the first eigenvector of $A$.

\begin{lemma}
\label{lem:T_flattens}
Let $W$ be the set of vertices $x$ for which $\deg x > 0,$
and let $\fuzz$ be as in Proposition~\ref{prop:key_structure}.
For each $\delta >0$ and $m \geq 0,$ there is a constant $C=C(\delta, m)$ sufficiently large so that if
\(p \geq \delta \log n / n\)
then
\[
\sup_{ \substack{ \|x\| = 1, \\ x^t T^{1/2}\one_{W} = 0 } } | x^t T^{-1/2} \one_{\fuzz^c} | \leq C\frac{\sqrt{n}}{d}
\]
with probability at least $1-C\exp(-md).$
\end{lemma}
\begin{proof}
As we have that $|\fuzz| < n/(100d)$ by Proposition~\ref{prop:key_structure}, it follows that
\[
|x^tT^{1/2}\one_{\fuzz}| \leq \| T^{1/2}\one_{\fuzz}\| \leq \sqrt{d|\fuzz|} = O\left(\sqrt{n}\right).
\]
Further, we have that
\(
x^tT^{1/2}\one_{\fuzz}
=
-x^tT^{1/2}\one_{\fuzz^c},
\)
and hence it suffices to show that
\[
\sup_{ \substack{ \|x\| = 1, \\ x^t T^{1/2}\one = 0 } }
\left|x^t( T^{-1/2} - T^{1/2}/d)\one_{\fuzz^c}\right|
 \leq C\frac{\sqrt{n}}{d}.
\]
Taking norms,
\[
\left|x^t( T^{-1/2} - T^{1/2}/d)\one_{\fuzz^c}\right|
\leq
\left\|( T^{-1/2} - T^{1/2}/d)\one_{\fuzz^c}\right\|.
\]
Squaring this norm, we get
\[
\left\|( T^{-1/2} - T^{1/2}/d)\one_{\fuzz^c}\right\|^2
=\sum_{v \in \fuzz^c}
\left(
\frac{1}{\sqrt{\deg v}} - \frac{\sqrt{\deg v}}{d}
\right)^2
\leq
\frac{M}{d^3} \sum_{v \in \fuzz^c} \left(
\deg v - d
\right)^2.
\]
Lemma~\ref{lem:T_distortion} completes the proof.
\end{proof}

\begin{pfofthm}{Theorem~\ref{thm:ergap_giant}}
  We show that we satisfy the conditions in Lemma~\ref{lem:deterministic}.
  In Lemma~\ref{lem:regularity}, we show the bounded degree condition.
  In Proposition~\ref{prop:adj}, we show the adjacency matrix condition.
  In Proposition~\ref{prop:key_structure}, we show the fuzz condition.
  Finally, in Lemma~\ref{lem:T_flattens}, we show the parallel eigenspaces condition.
  Summing the failure probabilities, the failure probability in Theorem~\ref{thm:ergap_giant} is the sum of $Cn \exp(-(2-\epsilon)d)$ from Proposition~\ref{prop:key_structure} and $C\exp(-md^{1/4}\log n)$ from Proposition~\ref{prop:adj}, with all other errors much smaller in magnitude.
  Without the condition that $p \geq (\tfrac 12 + \delta) \log n / n,$ for some $\delta>0,$ the failure probability in Proposition~\ref{prop:key_structure} is not in control.
\end{pfofthm}

We wish to now show the lower bounds for $\lambda(\tilde{G}).$  We will use Lemma~\ref{lem:deterministic2}, and this requires that we show:
\begin{proposition}
\label{prop:handles}
If $p = \omega(\sqrt{\log n}/n)$ and $p \leq \tfrac12\log n / n$ then with high probability, there are four distinct vertices $a,b,c,d$ in the giant component for which the degrees of $a$ and $d$ are at least $np/2,$ the degrees of $b$ and $c$ are $2,$ and the induced subgraph on $(a,b,c,d)$ is a path.
\end{proposition}

We first show by the second moment method that such four-tuples $(a,b,c,d)$ exist in the graph with high probability.  We then show that with high probability, the small components have maximal degree $o(np),$ and hence these four-tuples must have been part of the giant component.

\begin{lemma}
\label{lem:handles_exist}
Suppose that $p = \omega( 1 / n)$ and that $p \leq \tfrac12\log n / n.$
Then, with high probability, there are four-tuples $(a,b,c,d)$ for which the degrees of $a$ and $d$ are at least $np/2,$ the degrees of $b$ and $c$ are $2,$ and the induced subgraph on $(a,b,c,d)$ is a path.
\end{lemma}
\begin{proof}
Define the pair of events
\begin{align*}
A(a,b,c,d) &=
\{
a \leftrightarrow b \leftrightarrow c \leftrightarrow d, \deg b = \deg c = 2
\}~\text{ and } \\
B(a,b,c,d) &= A(a,b,c,d) \cap
\{
\deg a \geq np/2, \deg d \geq np/2
\}.
\end{align*}
Set $S$ to be the number of occurrences of $B,$ i.e.
\[
S = \sum_{ a,b,c,d } \one[B(a,b,c,d)],
\]
with the sum over ordered $4$-tuples of distinct vertices $(a,b,c,d).$
We need to show that $S > 0$ with high probability.

The probability of $A$ can be explicitly calculated as
\[
\Pr \left[
A(a,b,c,d)
\right] = p^3(1-p)^{2(n-3)}.
\]
Meanwhile, conditional on $A(a,b,c,d),$ the probability of $B(a,b,c,d)$ is exactly the probability of having two specific vertices of degree at least $np/2 - 1$ in $G(n-2,p).$  Set $Q = \Pr\left[ X \geq np/2 \right]$ where $X\sim \Binom(n,p).$  Note that as $np \to \infty,$ we have that $Q=1-o(1).$

Furthermore, as $np \to \infty$ we have that
\[
\Pr \left[
B(a,b,c,d) ~\middle\vert~ A(a,b,c,d)
\right] = Q^2(1-o(1)),
\]
simply by conditioning on the edge between $a$ and $d.$
By summing over all possible tuples, it follows that
\(
\Exp S = \Theta( nQ^2(np)^3e^{-2np} ) = \omega(1).
\)
\newcommand{\bvec}{
(b_i)_{i=1}^4
}
\newcommand{\avec}{
(a_i)_{i=1}^4
}

For the variance of $S$, we need to compute probabilities of the pairs $B( \avec ) \cap B( \bvec).$  Note that if $a_2 = b_2$ then the only way both can happen is if $a_i = b_i$ for all $i \in [4].$  Analogous conclusions hold if $a_2 = b_3$ or if $a_3 \in \{b_2,b_3\}.$   Thus, the only nontrivial way for the events $B( \avec)$ and $B( \bvec)$ to intersect is if
\begin{enumerate}
\item all $a_i$ and $b_i$ are distinct,
\item $a_1 = b_1$ and the rest are distinct,
\item $a_1 = b_1,$ $a_4=b_4,$ and the rest are distinct, or
\item $a_i = b_i$ for all $i.$
\end{enumerate}
Note that there's no need to consider $a_1 = b_4,$ as the event $B( \bvec )$ is preserved under reversing the $a_i.$  Likewise, there's no need to consider $a_4=b_4,$ as one can reverse both $a_i$ and $b_i.$  Set $T_i$ to be the pairs of tuples satisfying each of the $4$ cases.

If the pair is in $T_1,$ then
\[
\Pr \left[
B( \avec) \cap
B( \bvec)
\middle\vert
A( \avec) \cap
A( \bvec)
\right]
=Q^4(1-o(1))
\]
as once more, this is the statement that four vertices in $G(n-4,p)$ have degree at least $(np/2-1)$.  We also have that
\[
\Pr \left[
A( \avec) \cap
A( \bvec)
\right] = p^6(1-p)^{4n-16},
\]
so that
\[
\Pr \left[
B( \avec) \cap
B( \bvec)
\right] =
\Pr \left[
B( \avec)
\right]^2(1-o(1)).
\]
Thus the contribution of the pairs in $T_1$ to the variance of $S$ is $o( (\Exp S)^2).$

For terms from $T_2,$ the same reasoning as above shows that
\[
\Pr \left[
B( \avec) \cap
B( \bvec)
\right] = Q^3p^6(1-p)^{4n}(1-o(1))
\]
For such pairs, however, we have that $|T_2| = \Theta(n^7),$ and hence the contribution to the variance of $S$ is $o( (\Exp S)^2).$  In the same way, the contributions of $T_3$ and $T_4$ are smaller still.  As each is individually of order $o( (\Exp S)^2 )$, we have that $S>0$ with high probability.
\end{proof}

\begin{lemma}
\label{lem:giant_component_size}
Suppose that $p = \omega(1/n),$ then for any $\epsilon >0,$ the number of vertices not in the giant component is at most $n e^{-(1-\epsilon)np}$ with high probability.
\end{lemma}
\begin{proof}
Set $R$ to be the number of vertices not in the largest component of $G(n,p)$.  If $W$ is the set of these vertices, then $W$ satisfies $e(W,W^c)=0.$  Therefore, if there is no collection $W$ of at least $r$ vertices such that $e(W,W^c)=0,$ then $R < r.$ 

The expected number $\Exp N_r$ of such collections $W$ is given by
\[
\Exp N_r
= (1-p)^{r(n-r)} { n \choose r }.
\]
Set $r_0 = n e^{-(1-\epsilon)np}.$  We will show that
\(
\sum_{r=r_0}^{n/2} \Exp N_r \to 0,
\)
which implies the lemma.

Subdivide the sum into two pieces $S_1$ and $S_2,$ given by
$S_1 = \sum_{r_0}^{\lfloor \epsilon n/4 \rfloor} \Exp N_r$
and
$S_2 = \sum_{\lfloor \epsilon n/4 \rfloor}^{n/2} \Exp N_r.$ For $\lfloor \epsilon n/4 \rfloor \leq r \leq n/2,$
\[
\Exp N_r
= (1-p)^{r(n-r)} { n \choose r }
\leq e^{-c_\epsilon n^2p} 2^n,
\]
for some $c_\epsilon > 0,$ which decays exponentially in $n$ as $np \to \infty.$  Hence $S_2 \to 0.$

As for $S_1,$ we claim that for any $\alpha > 0$ there is an $n \geq n_0(\alpha,\epsilon)$ sufficiently large so that for all $r_0 < r < \epsilon n/4,$ $\Exp N_{r+1} \leq \alpha \Exp N_r$ for all $n \geq n_0(\alpha,\epsilon).$  Estimating for these $r,$
\begin{align*}
\frac{\Exp N_{r+1}}{\Exp N_r}
&= (1-p)^{n-2r+1} \frac{n-r-1}{r+1} \\
&\leq \frac{ne^{-np+2rp} }{r}.\\
&\leq \frac{ne^{-(1-\epsilon/2)np} }{r}.\\
&\leq e^{-\epsilon np/2}.
\end{align*}
Hence, as $np \to \infty,$ this is eventually less than any positive $\alpha.$

As $S_1$ is dominated by a geometric series, and $S_1 = O( \Exp N_{r_0}).$  For this leading term, we get that
\[
\Exp N_{r_0} \leq e^{-pr_0(n-r_0)} \left( \frac{en}{r_0}\right)^{r_0}
\leq \exp\left( -\epsilon n^2pe^{-(1-\epsilon)np}(1-o(1))
\right) \to 0,
\]
completing the proof.
\end{proof}

\begin{lemma}
\label{lem:dust_degree_bound}
If $p = \omega( \sqrt{\log n}/ n),$ then with high probability, the maximum degree of the vertices not in the giant component is at most $np/100.$
\end{lemma}
\begin{proof}

Set $R$ to be the number of vertices not in the giant component.  By Lemma~\ref{lem:giant_component_size}, we have that $R \leq ne^{-np/2}$ with high probability.  Suppose that $W$ is a fixed collection of vertices of size $r.$  Conditional on there being no edges between $W$ and $W^c,$ the law of the induced graph on $W$ is simply that of $G(r,p).$

Let $X \sim \Binomial(r-1,p).$  Then by Lemma~\ref{maybe} there are absolute constants $c>0$ and $M > 0$ so that
\[
\Pr \left[ X > np/100 \right]
\leq \exp(-cnp\log(n/r))
\]
provided $r<n/M.$  Setting $E_W$ to be the event that $W$ and $W^c$ are not connected
\[
\Pr \left[
\max_{w \in W} \deg w > np/100 ~\middle\vert~ E_W
\right] \leq r\exp(-cnp\log(n/r)).
\]
Let $Y$ be the max degree of all vertices not in the largest component
As the previous bound holds for all $W$ in consideration, we get that
\[
\Pr \left[
Y > np/100 ~\middle\vert~ R = r
\right] \leq r\exp(-cnp\log(n/r)).
\]
This bound is monotone increasing in $r,$ and so we get that
\[
\Pr \left[
Y > np/100 ~\middle\vert~ R \leq ne^{-np/2}
\right] \leq n\exp(-c(np)^2(1-o(1)))
\]
for some absolute constant $c$.
Thus by the assumption on $np,$ the desired claim holds.
\end{proof}

\begin{pfofthm}{
Theorem~\ref{thm:ergap_giant2} and
Proposition~\ref{prop:handles}
}

For Proposition~\ref{prop:handles}, the previous three Lemmas~\ref{lem:handles_exist},~\ref{lem:giant_component_size}, and~\ref{lem:dust_degree_bound} show the desired claim that w.h.p.\ there are tuples $(a,b,c,d)$ of vertices in the giant component for which $\deg a$ and $\deg d$ are at least $np/2,$ vertices $b$ and $c$ have degree $2,$ and the induced graph on these vertices is a path.

Letting $H$ be the giant component of the graph, then there is a constant $C$ so that the eigenvalues of the Laplacian of $H$ satisfy
\[
\lambda_{|H|} \geq \tfrac{3}{2}
\]
and
\[
\lambda_{2} \leq \tfrac{1}{2} + C/\sqrt{np},
\]
by Lemma~\ref{lem:deterministic2}.
\end{pfofthm}

\section{Gap process theorem} \label{sec:proc}
%%NOTE: I use k instead of d here to avoid confusion with d=(n-1)p.
In this section we prove a general process-version theorem for the spectral gap below the connectivity threshold.  We recall the definition of $\LMP[t][k],$ the continuous time Linial-Meshulam process.  Let $F_k$ denote the collection of all possible $k$--faces on $n$ vertices, and let $F_k(S)$ for simplicial complex $S$ be all $k$--faces of $S$.  Let $\{T_\sigma, \sigma \in F_k\}$ be an i.i.d.\ family of $\Exponential(1)$ variables.  Define $\{\LMP[t][k], t \geq 0\}$ to be the continuous time Markov process
where $\LMP[0][k]$ is the complete $(k-1)$-skeleton of the $n$-simplex and its $k$-faces are given by
\[
F_k(\LMP[t][k]) = \{ \sigma \in F_k~:~T_\sigma \leq t\}.
\]
Thus $\LMP[t][k]$ is the complex whose $k$-faces have been born up to time $t,$ and $\LMP[\infty][k]$ is the complete $k$-skeleton of the $n$-simplex.  For $k=1,$ this recovers the standard continuous time \ER process.  For fixed $t,$ $\LMP[t][k]$ is the Bernoulli complex $Y_k(n,p(t))$ with $p(t) = 1 - e^{-t}.$  Let $d(t) = (n-1)p(t).$ Fix $\delta \in (0, \tfrac12)$ and define $t_0$ by the relation that
\[
p(t_0) = \begin{cases}
(\tfrac 12 + \delta) \log n /  n & k = 1, \\
(k - 1 + \delta) \log n /  n & k > 1.
\end{cases}
\]

For any $(k-2)$--dimensional face $f$ of a $k$--dimensional simplicial complex $S,$ we identify its \emph{link} with a graph, denoted $\link(f).$ 
We will only consider links of $(k-2)$--dimensional faces.
This graph $\link(f)$ has vertex set given by all $(k-1)$--dimensional faces containing $f.$  Two of these edges $e$ and $g$ are connected if and only if $e \cup g,$ which is a $k$--dimensional face, is contained in $S.$  

For example, when $k=1$ and $S$ is a graph, the only $(k-2)$--dimensional face is the empty set.  Its link has vertex set given by all $0$--dimensional faces (all vertices), and vertices are connected if and only if they are contained in an edge.  Hence, in this case $\link(\emptyset)$ can be identified with the original graph $S.$ 

In $\LMP[t][k],$ which has a complete $(k-1)$--skeleton, each link is distributed as a $G(n-k+1,p(t)).$  These links $\{\link(f)\},$ where $f$ ranges over all $(k-2)$--dimensional faces, are not independent, and in fact are analysis rests in some ways on exploiting their exact dependency structure.

Recall that we refer to a $(k-1)$--dimensional face $f$ as \emph{isolated} if and only if it is not contained in any $k$--dimensional face.  Note that a face $f$ is isolated if and only if it is an isolated vertex in $\link(g)$ for all $(k-2)$--dimensional $g \subset f.$

\begin{theorem}
\label{thm:stopping_time}
%For each $\link(f),$ let $I_f$ denote the number of isolated vertices.
Let $\tLMP[t][k]$ denote the process 
derived from $\LMP[t][k]$
by removing every isolated $(k-1)$-face.
There is a constant $C=C(k,\delta)$ so that
with high probability
the normalized Laplacian of $\link(f)$ of
every dimension--$(k-2)$ face $f$ of
\(
{\tLMP[t][k]}
\)
has
\[
\max_{i > 1} \left| 1 - \lambda_i\right| < \frac{C}{\sqrt{d(t)}}.
\]
for all $t \geq t_0.$
\end{theorem}
An equivalent formulation is that each $\link(f)$ for codimension-$2$ $f \in \LMP[t][k]$ consists of isolated vertices and a giant component whose gap is $1-C/\sqrt{d(t)}$ for all time $t \geq t_0.$
In the higher-dimensional setting, the proof is more complicated than simply studying each link individually and taking the union bound.  The key is to study the ``fuzz''  globally.  To this end, for each $\link(f)$ and for any $M \geq 1,$ let
\begin{equation}
\label{eq:fuzz_t}
\fuzz[f][t]
=\{
w \in \vertexsetof(\link(f))~:~ \deg_{\link(f)}(w) \leq d(t_0)/M
\}.
\end{equation}
Note that this makes each $\fuzz[f][t]$ monotone decreasing in $t.$

\begin{lemma}
\label{lem:supersmall}
There is an $M=M(k,\delta)$ and an $\epsilon=\epsilon(k,\delta)$ so that
\[
\sum_{f \in F_{k-2}}\left|
\fuzz[f][t_0]
\right|^2
\leq n^{1 - \epsilon}
\]
with overwhelming probability.
\end{lemma}
\begin{proof}
For $k=1,$ there is only one link to consider, and so it suffices to show that
\(
\left|\fuzz[\emptyset][t_0]\right| \leq n^{1/2 - \epsilon}.
\)
For $k>1,$ we proceed by showing that for any $\epsilon$ there is an $M$ so that both
\begin{enumerate}
\item
\(
\max_{f \in F_{k-2}} \left|
\fuzz[f][t_0]
\right|
\leq n^{\epsilon}
\)
\item
\(
\sum_{f \in F_{k-2}}\left|
\fuzz[f][t_0]
\right|
\leq n^{1 - 2\epsilon}
\)
\end{enumerate}
hold with overwhelming probability.

The first condition follows from an identical argument to the first part of Proposition~\ref{prop:key_structure}; the $k=1$ case follows from an identical argument, and we just sketch the $k>1$ case.  As before, for any $1 > \eta > 0,$ there is an $M(\delta,\eta)$ sufficiently large so that for a fixed set of vertices $w_1,w_2,\ldots,w_{\lceil {n}^{\epsilon} \rceil},$
\[
\Pr \left[
\deg_{\link(f)}(w_i) \leq d(t_0)/M,~\forall~1 \leq i \leq\lceil n^{\epsilon} \rceil
\right] = O(\exp(-{n}^{\epsilon}d(t_0)(1-\eta))).
\]
This overwhelms the $O(\exp( (1-\epsilon)n^{\epsilon}\log n))$ possible choices of vertices as
\[
d(t_0)/\log n > (1+\delta)(1+o(1))
\] and $\eta$ may be chosen sufficiently small.  As there are only $O(n^{k-1})$ many links to consider, this may be taken to hold for all links simultaneously with overwhelming probability.

We now turn to the second condition.
For a fixed $(k-1)$-dimensional face $f$, let $X_f$ denote the number of $k$-faces in $\LMP[t_0][k]$ containing $f.$  If $f$ is a vertex in a $(k-2)$-dimensional face of $\LMP[t_0][k],$ then $X_f$ is the degree of that vertex in $\link(f).$  Hence
\[
\frac{1}{k}
\sum_{f \in F_{k-2}}\left|
\fuzz[f][t_0]
\right|
=
\sum_{f \in F_{k-1}}
\one[X_f \leq d(t_0)/M].
\]
Thus by adjusting $\epsilon,$ it suffices to show the claim for the right hand side.
Call a collection $S$ of $(k-1)$-faces \emph{balanced} if
\[
\max_{w \in F_{k-2}} \left| \{
\sigma \in S~:~ w \subset \sigma
\}\right| \leq n^{\epsilon}.
\]
Observe that we have shown that with overwhelming probability the set 
\[
  S = \left\{ f \in F_{k-1} : X_f \leq d(t_0)/M \right\}
\]
is balanced with overwhelming probability.

By symmetry we have
\begin{multline*}
\Pr\left[
  \exists~f_1, f_2,\ldots,f_r \in F_{k-1}~:~ X_{f_i} \leq d(t_0)/M,~1\leq i \leq r,~\{f_i\}~\text{balanced}
\right] \\
\leq
{ {n \choose k} \choose r}
\Pr\left[
X_{f_i} \leq d(t_0)/M, 1 \leq i \leq r,~\{f_i\}~\text{balanced}
\right].
\end{multline*}
Let $X$ denote the number of $k$-faces that contain some $f_i.$  If every $X_{f_i} \leq d(t_0)/M,$ it follows that $X \leq rd(t_0)/M.$  Each $f_i$ is contained in $n-k$ possible $k$-faces, but it may be possible that some $f_i$ and $f_j$ are both contained in a single $k$-face.  If this occurs, however, it must be that $|f_i \cap f_j|=k-1.$ In other words, each contains a common $(k-2)$-face.  Furthermore, there is at most one $k$-face that contains both $f_i$ and $f_j.$

A fixed face $f_j$ contains $k$ distinct $(k-2)$-faces $q_1,q_2,\ldots,q_k.$
As $\{f_i\}$ is balanced, each $q_l$ is contained in at most $n^{\epsilon}$ distinct $f_i.$  Thus there are at most $n^{\epsilon}k$ many $k$-faces that contain $f_j$ and some other $f_i,$ and this implies there are at least $r(n-k-n^{\epsilon}k)$ distinct possible $k$-faces that contain some $f_i.$ It follows that $X$ stochastically dominates a
\(
\Binomial\left( \bigl\lceil r(n-k-n^{\epsilon}k)\bigr\rceil, p(t_0)\right)
\)
variable.  Applying Lemma~\ref{magic}, we get
\begin{multline*}
\Pr\left[
X_{f_i} \leq d(t_0)/M, 1 \leq i \leq r,~\{f_i\}~\text{balanced}
\right]
\leq
\Pr\left[
X \leq rd(t_0)/M
\right]
\\
\leq
\exp\left(
-r(n-k-n^{\epsilon}k)p(t_0) + \tfrac{rd(t_0)}{M}(1 + \log\tfrac{M(1+r(n-k-n^{\epsilon}k))p(t_0)}{rd(t_0)}) 
\right).
\end{multline*}
Thus, we get
\begin{multline*}
\log \Pr\left[
\exists~f_1, f_2,\ldots,f_r~:~ X_{f_i} \leq d(t_0)/M,~1\leq i \leq r
\right] \\
\leq r\left[
(k\log n - \log r)-d(t_0) + \frac{d(t_0)}{M}( 1 + \log(M))
\right](1+o(1)).
\end{multline*}
Since $d(t_0) \geq (k - 1 + \delta) \log n - o(1),$ we can set $r = [n^{1 - \delta/2}]$ and make $M$ sufficiently large that
\[
(k\log n - \log r)-d(t_0) + \frac{d(t_0)}{M}( 1 + \log(M))
\to -\infty.
\]
Taking $\epsilon = \delta/4,$ we have shown the desired claim.
\end{proof}

With global control on the number of exceptional vertices, the proof now reduces to essentially a union bound over all later times and links.

\begin{lemma}
\label{lem:proc_easy}
There is a constant $C=C(k)$ so that with high probability,
every $\link(f)$ where $f \in \LMP[t][k]$ has dimension $(k-2)$ satisfies
\begin{enumerate}
\item {\bf Bounded degree condition (b.d.c)} Every vertex has degree at most $C d(t).$
\item {\bf Adjacency matrix} The adjacency matrix of the link satisfies
\[
\sup_{ \substack{ \|x\| = 1, x^t \one = 0 \\ \|y\| = 1 } } | x^tAy| \leq C\sqrt{d(t)}.
\]
\item {\bf Parallel eigenspaces} Setting $\fuzz = \fuzz[f][t]$ and $T$ to be the diagonal matrix of degrees of the link,
\[
\sup_{ \substack{ \|x\| = 1, \\ x^t T^{1/2}\one_{W} = 0 } } | x^t T^{-1/2} \one_{\fuzz^c}| \leq C\frac{\sqrt{n}}{d(t)}.
\]
\end{enumerate}
for all $t \geq t_0.$
\end{lemma}
\begin{proof}
Let $I$ be the interval $[t_1,t_2],$ where $t_0 \leq t_1 \leq t_2.$  The probability that there are two faces that appear in this interval can be bounded by
\[
\Pr \left[
\exists~\sigma_1, \sigma_2~:~T_{\sigma_1} \in I \text{ and } T_{\sigma_2} \in I
\right] \leq { n \choose k}^2 \left( p(t_2)  - p(t_1) \right)^2.
\]
Let $r$ be the smallest integer so that $p(t_0) + rn^{-2k-1} \geq 1.$
Set $p_i = p(t_0) + i n^{-2k-1}$ for all $0 \leq i < r,$
and set $p_r = 1.$
Let $t_i$ be such that $p(t_i) = p_i,$ and set $t_r = \infty.$
Note that for $t \in [t_i, t_{i+1}),$ $\LMP[t][k] \neq \LMP[t_i][k]$ and $\LMP[t][k] \neq \LMP[t_{i+1}][k]$ implies there must be two faces $\sigma_1$ and $\sigma_2$ for which $T_{\sigma_1}, T_{\sigma_2} \in [t_i,t_{i+1}).$  Hence,
\begin{align*}
\Pr \left[
\exists ~t \geq t_0 ~:~ \LMP[t][k] \neq \LMP[t_i][k]~\forall~ 0 \leq i \leq r
\right]
&\leq \sum_{i=0}^{r-1}
\Pr \left[
\exists~\sigma_1, \sigma_2~:~T_{\sigma_1}, T_{\sigma_2} \in I
\right] \\
&\leq \sum_{i=0}^{r-1}
n^{-2k-2}
\leq n^{-2}.
\end{align*}

By applying Lemma~\ref{lem:regularity}, Proposition~\ref{prop:adj}, and Lemma~\ref{lem:T_flattens} with $m$ sufficiently large, we may thus assure
that there is a constant sufficiently large
that these properties occur
for all links of all $\LMP[t_i][k],$ for $0 \leq i \leq {r-1}.$
\end{proof}

\begin{lemma}
\label{lem:fuzz_process}
There is an $M=M(k,\delta)$ and a constant $C=C(M,k)$ so that with $t_1$ satisfying $p(t_1) = C \log n / n,$ all
\(
\fuzz[f][t] = \emptyset
\)
for $t \geq t_1$ with high probability.
Further, for all $t_1 \geq t \geq t_0$ every $\link(f)$ of $\LMP[t][k]$ satisfies
\begin{enumerate}
\item $|\fuzz| \leq \frac{n}{2},$
\item $\fuzz$ is an independent set,
\item and \( \max_{u \in \fuzz^c} \dset{u}{\fuzz} \leq 1\)
\end{enumerate}
with $\fuzz=\fuzz[f][t].$
\end{lemma}
\begin{proof}
There is an $M_1$ so that this holds for $\LMP[t_0][k]$ by Proposition~\ref{prop:key_structure} and by taking the union bound over all links.  Likewise, there is an $M_2$ so that the conclusions of Lemma~\ref{lem:supersmall} holds.  Take $M$ to be the maximum of these, and note that from monotonicity, the conclusions of both the proposition and lemma hold.  As $\fuzz[f][t]$ is monotone in $t$ also, we have that
\[
|\fuzz[f][t]| \leq
|\fuzz[f][t_0]| \leq
n/2
\]
is satisfied for all $n$ sufficiently large.

From a union bound and Lemma~\ref{magic}, we may choose $C=C(M,k)$ sufficiently large so that with probability going to $1,$
\[
\fuzz[f][t_1] = \emptyset
\]
for all $f \in F_{k-2}.$

Let $\tau_i$ be the times at which the $i^{th}$ face is added to $\LMP[t][k]$ after time $t_0,$ and let $\tau_0 = t_0.$  Likewise, let $\Delta_i$ denote the $i^{th}$ face, and let $\filt(\tau_i) = \sigma(\LMP[\tau_i][k]).$   Let $N$ denote the largest $i$ so that $\tau_i \leq C\log n / n.$ From Chernoff bounds, there are at most $100 C (\log n) n^{k}$ many $k$-dimensional faces in $\LMP[t_1][k]$ with overwhelming probability, and hence $N \leq 100 (\log n) n^k$ with overwhelming probability.

We begin by bounding the probability that a newly added face creates an edge between two vertices of $\fuzz[f][t]$ for some $f \in F_{k-2}.$
\begin{align}
\label{eq:indset}
\Pr\left[
\exists~ u,v \in \fuzz[f][\tau_i]~:~u,v \in \Delta_{i+1} \middle \vert \filt(\tau_i)
\right]
&\leq
\frac{
| \fuzz[f][\tau_i] |^2
}{
|F_k| - |\LMP[\tau_i][k]|
} \\
\notag
&\leq
\frac{
| \fuzz[f][\tau_0] |^2
}{
|F_k| - |\LMP[t_1][k]|
}.
\end{align}
Let $E_{i,f}$ denote the event that
\begin{enumerate}
\item
  the number of $k$--dimensional faces
\(
|\LMP[t_1][k] | \leq 100 C n^k\log n,
\)
\item
\(
\sum_{f \in F_{k-2}}
| \fuzz[f][t_0] |^2 \leq n^{1 - \epsilon},
\)
\item there exists $u$ and $v$ in $\fuzz[f][\tau_i]$ so that $u \in \Delta_{i+1}$ and $v \in \Delta_{i+1}.$
\end{enumerate}
By conditioning, we have that
\begin{align*}
\Pr
\left[
\cup_{i,f} E_{i,f}
\right]
&\leq
\Exp
\sum_{i=0}^N
\sum_{f \in F_{k-2}}
\frac{
| \fuzz[f][\tau_0] |^2 \one[E_{i,f}]
}{
|F_k| - |\LMP[t_1][k]|
} \\
&\leq
\Exp
\sum_{i=0}^N
\frac{
\sum_{f \in F_{k-2}}
| \fuzz[f][\tau_0] |^2 \one[E_{i,f}]
}{
|F_k| - 100C n^k \log n
} \\
&\leq
\Exp
\sum_{i=0}^N
\frac{
n^{1-\epsilon}\one[\LMP[t_1][k] \leq 100Cn^k\log n]
}{
|F_k| - 100C n^k \log n
} \\
&\leq
\frac{
{(100Cn^k\log n)}
n^{1-\epsilon}
}{
|F_k| - 100C n^k \log n
} = O(n^{-\epsilon} \log n).
\end{align*}
Thus with high probability, no face added between $t_0$ and $t_1$ creates an edge between two elements of any $\fuzz[f][t].$

We now turn to bounding the probability that a newly added face connects an element of $\fuzz[f][t]$ to a neighbor of $\fuzz[f][t].$  Let $\mathcal{N}_M^f(t)$ be the set of neighbors of $\fuzz[f][t],$ and let $D(t)$ be an upper bound for the degree of a vertex of any link of $\LMP[t][k].$  Note that
\(
|\mathcal{N}_M^f(t)| \leq D(t) | \fuzz[f][t] |.
\)
Then
\begin{align*}
\Pr\left[
\exists~ u\in \fuzz[f][\tau_i],
v\in \mathcal{N}_M^f(t)
~:~u,v \in \Delta_{i+1} \middle \vert \filt(\tau_i)
\right]
\leq
\frac{
D(\tau_i)| \fuzz[f][\tau_i] |^2
}{
|F_k| - |\LMP[\tau_i][k]|
}.
\end{align*}
With high probability, there is a constant $K$ so that all the degrees can be bounded by $K \log n$ for all $t \leq t_1.$  This failure probability is at most a logarithmic factor more than the failure probability in~\eqref{eq:indset}.  Hence the same proof shows that with high probability, no added face increases
\[
\max_{u \in \vertexsetof(\link(f)) \setminus \fuzz[f][t]} e(u,\fuzz[f][t]).
\]
\end{proof}

\begin{pfofthm}{Theorem~\ref{thm:stopping_time}}
We are essentially ready to apply Lemma~\ref{lem:deterministic}.  The only concern is that in~\eqref{eq:fuzz_t}, the set $\fuzz[f][t]$ is defined in terms of $d(t_0)$ and not $d(t).$  However, as noted in Lemma~\ref{lem:fuzz_process}, all these sets disappear once $p(t_1)=C\log n / n,$ at which point $d(t)$ has only risen by a factor of $K=\frac{p(t_1)}{p(t_0)}.$  Thus,
\[
Q^f(t) = \{
w \in \vertexsetof(\link(f))~:~ \deg_{\link(f)}(w) \leq d(t)/KM
\}
\subseteq \fuzz[f][t],
\]
for all $t \leq t_1,$ and by monotonicity, all the desired properties of $\fuzz[f][t]$ transfer to $Q^f(t).$
Thus Lemmas~\ref{lem:proc_easy} and~\ref{lem:fuzz_process} show all the needed properties of
Lemma~\ref{lem:deterministic} hold, completing the proof.
\end{pfofthm}

\section{Cohomology structure theorem} \label{sec:cohom}

The structure theorem for cohomology relies on the following theorem of \BS~\cite{BS}.  A simplicial complex $\Delta$ is called \emph{pure} $k$--dimensional if it is $k$--dimensional and every face is contained in a $k$--dimensional one.

\begin{BSC} \label{thm:BStool}
If \( \Delta \) is a finite, pure \(k\)-dimensional simplicial complex,
so that for every $(k-2)$-dimensional face $\sigma,$
the normalized Laplacian $L= L [\link(\sigma) ]$ satisfies
\(
\lambda_2  > 1- \frac{1}{k}
\)
then \(H^{k-1}(\Delta, \mathbb{Q}) =0.\)
\end{BSC}

\begin{proof}[Proof of Theorem~\ref{thm:Hhit}]
  Recall that we define $t_0$ so that $p(t_0) = (k-1+\delta)\log n / n.$  Let $\tilde{Y_t}$ denote the simplicial complex $\LMP[t][k]$ with all its isolated $(k-1)$-faces deleted.
By Theorem~\ref{thm:stopping_time}, w.h.p. for all $t \geq t_0$, all links of $\tilde{Y_t}$ have $\lambda_2(L) = 1-o(1).$

We need to check that $\tilde{Y_t}$ is pure $k$-dimensional, i.e. that every face is contained in some $k$-dimensional face.  Note that this can only fail if there is some $(k-2)$-dimensional face of $\LMP[t][k]$ that is not contained in any $k$-dimensional face.  As this is a monotone property, it suffices to check that $\LMP[t_0][k]$ has no such $(k-2)$-faces.

Put $I$ to be the number of isolated $(k-2)$-faces in $\LMP[t_0][k].$  Then
\[
\Exp I = {n \choose k-1} (1-p(t_0))^{ n^2/2}(1-o(1)),
\]
which decays exponentially in $n.$  Hence, $\tilde{Y_t}$ is pure $k$-dimensional w.h.p.\ for all $t \geq t_0$, and so Theorem~\ref{thm:BStool} applies.  It follows that $H^{k-1}( \tilde{Y_t}, \Q) = 0,$ and it remains to compare $H^{k-1}( \tilde{Y_t}, \Q)$ and $H^{k-1}( \LMP[t][k], \Q).$

For what remains, fix $t \geq t_0.$
It will follow from induction that each additional $(k-1)$-face we glue to $\tilde{Y}$ increases the dimension of the $(k-1)$-cohomology by $1.$  Let $Z$ be the complex formed by including one of the isolated $(k-1)$-faces of $Y$ back into $\tilde{Y}.$  Let $B$ be a neighborhood of the included $(k-1)$-face that is homotopic to a single $(k-1)$-simplex.  Then the Mayer-Vietoris sequence (see Chapter 3 of~\cite{Hatcher}) for the $(k-1)$-dimensional cohomology is
\[
\cdots
\rightarrow
H^{k-1}(Z,\Q)
\rightarrow
H^{k-1}(\tilde{Y},\Q) \oplus
H^{k-1}(B,\Q)
\rightarrow
H^{k-1}(\tilde{Y} \cap B,\Q)
\rightarrow
H^{k}(Z,\Q).
\]
As \(\tilde{Y} \cap B\) is homotopic to a $(k-2)$-dimensional sphere,
\( H^{k-1}(\tilde{Y} \cap B,\Q) = 0 \).  Also,
\(H^{k-1}(B,\Q) = \Q \), and so this sequence becomes
\[
0
\rightarrow
H^{k-1}(Z,\Q)
\rightarrow
H^{k-1}(\tilde{Y},\Q) \oplus \Q
\rightarrow
0,
\]
or otherwise stated,
\( H^{k-1}(Z,\Q) \cong
H^{k-1}(\tilde{Y},\Q) \oplus \Q \).  Each additional isolated $(k-1)$-faces increases the dimension by one by the very same argument, which completes the proof.
\end{proof}

\section{Property (T)} \label{sec:propT}

The proof here is nearly identical to the proof of the cohomology vanishing structure theorem.
To establish our results concerning property (T) of random fundamental groups, we will use the following theorem of \.Zuk.

\begin{Zuk} \label{thm:tool}
If $X$ is a pure $2$-dimensional locally-finite simplicial complex so that for every vertex $v$, the vertex link $\link(v)$ is connected and the normalized Laplacian $L= L [\link(v) ]$ satisfies $\lambda_2(L) > 1/2$, then $\pi_1(X)$ has property~(T).
\end{Zuk}

\begin{proof}[Proof of Theorem \ref{thm:structure}]

%Suppose that $$ p \ge \frac{ ( 1 + \delta) \log n}{n},$$
%and consider the link of a vertex $v$ in $Y \sim Y_2(n,p)$. Such a vertex link has the same distribution as a random graph
%$G \sim G(n-1, p)$. Since $\delta > 0$, we may choose $\epsilon > 0$ depending on $\delta$ so that $(2-\epsilon) (1+ \delta) > 2$. The average degree $d$ satisfies $d \ge np \ge ( 1 + \delta - o(1) ) \log n $.
%So by \ref{thm:ergap_giant}, with probability at least
%$$ 1- n \exp \left( -( 2- \epsilon) ( 1 + \delta - o(1)) \log n  \right) \ge 1 - n^{-C},$$
%for some constant $C > 1$,
%the link of $v$ consists of only isolated vertices, and a giant component with smallest positive eigenvalue $1 - o(1)$.

%Since $C > 1$ we may apply a union bound, and w.h.p.\ the same structure holds for all $n$ vertex links.

%We consider $Y_2(n,m)$ where \(m \geq m_0\), and \(m_0 \sim {n \choose 3}\left(\frac{(1+\delta)\log n}{n}\).

%Recall that we define $t_0$ so that $p(t_0) = (1+\delta)\log n / n.$
%By Theorem~\ref{thm:stopping_time}, w.h.p. for all $t \geq t_0$, all links of $\LMP$ consist of isolated vertices and a connected component with $\lambda_2(L) > 1/2$ once $n$ is sufficiently large.  So for the remainder, fix $t$ and set $Y=\LMP.$

%Now let $\tilde{Y}$ denote the simplicial complex $Y$ with all its isolated edges deleted. This deletes isolated vertices from the vertex links, but otherwise leaves them unaffected. So all the vertex links in $\tilde{Y}$ are connected, and have spectral gaps $1-o(1)$. Then by \.Zuk's criterion, $\pi_1(\tilde{Y})$ has property~(T).

  Recall that we define $t_0$ so that $p(t_0) = (1+\delta)\log n / n.$  Let $\tilde{Y_t}$ denote the simplicial complex $\LMP[t][2]$ with all its isolated edges deleted.
By Theorem~\ref{thm:stopping_time}, w.h.p. for all $t \geq t_0$, all links of $\tilde{Y_t}$ have $\lambda_2(L) = 1-o(1).$  Then by \emph{\.Zuk's criterion}, $\pi_1(\tilde{Y_t})$ has property~(T) for all $t \geq t_0.$

Fix $t \geq t_0.$
It only remains to compare the fundamental groups $\pi_1(\tilde{Y})$ and $\pi_1(Y)$. But attaching a $1$-cell to a connected CW complex $W$ adds a free $\Z$-factor to the fundamental group $\pi_1(W)$, by the Seifert--van Kampen theorem (see Theorem 1.20 of~\cite{Hatcher}). So we only need to check that deleting all the isolated edges in $Y$ does not result in a disconnected complex $\tilde{Y}$.

Removing less than $n-1$ edges from the complete graph $K_n$ can not disconnect it; indeed, to separate a component of order $k$ form the rest of the graph requires removing at least $k(n-k)$ edges, which is minimized when $k=1$.  Thus we need only check that the number of isolated edges is fewer than $n-1.$  From monotonicity, it suffices to show that at time $t_0$ the number of isolated edges is w.h.p. $o(n).$  

By linearity of expectation, the expected number of edges deleted $\expect[D]$ is given by
\begin{align*}
\expect[D] &= { n \choose 2} (1 - p(t_0) )^{n-2}\\
& \le \frac{1}{2} n^2 \exp ( -p(t_0) (n-2) )\\
& \le O \left( n^{1-c} \right)
\end{align*}
for some constant $c > 0$.  By the second moment method, for example, $D$ is tightly concentrated around its mean, so w.h.p.\ $\tilde{Y}$ is connected. The claim follows.
\end{proof}

Corollary \ref{cor:Twin} quickly follows.

\begin{proof}[Proof of Corollary \ref{cor:Twin}]
Let $I$ denote the number of isolated edges. The expected number of isolated edges $\expect[I]$ is
\begin{align*}
\expect [i] &= {n \choose 2} (1 - p)^{n-2} \leq n^2 e^{-np}
\end{align*}
Taking $p = (2\log n + f(n))/n,$ where $f(n) \to \infty,$ this is seen to go to $0,$ completing the proof.
\end{proof}

\section{Kahn-Szemer\'erdi argument}
\label{sec:KSz}
In this section we give the proof of Proposition~\ref{prop:adj} and Lemma~\ref{lem:regularity}, which are minor modifications of the standard Kahn--Szemer\'erdi argument.

We begin with a proof of the regularity conditions.
\begin{pfofthm}{Lemma \ref{lem:regularity}}
For any vertex $v,$
$\deg(v)$ is a binomial random variable with mean $d>\delta \log(n)$. By Lemma \ref{maybe},
\(
\prob(\deg(v)>\csubzero  d) \leq \exp\left(-\tfrac{d \csubzero  \log \csubzero }{3}\right)
\)
provided $\csubzero > 4.$
Thus taking the union bound over all vertices, we get that
\[
\Pr\left[
\bdc \text{ fails}
\right] \leq \exp( d (\tfrac{1}{\delta}-\tfrac{\csubzero  \log \csubzero }{3})).
\]
By taking $\csubzero$ sufficiently large, we may take
\[
\frac{1}{\delta}-\frac{\csubzero  \log \csubzero }{3} \leq -m,
\]
completing the proof of the first claim.

We will now turn to showing the discrepancy property, for which we need to show there are constants $c_i=c_i(\delta,m)$ so that at least one of
\begin{enumerate}
\item
$
\tfrac{e(A,B)}{\mu(A,B)} \leq \csubone
$
\item
$
e(A,B)\log \tfrac{e(A,B)}{\mu(A,B)} \leq \csubtwo( |A| \vee |B|)\log \tfrac{n}{|A| \vee |B|}
$
\item
$
|A| \vee |B| \leq d^{1/4}/100
$
\end{enumerate}
Note that these properties are monotone in $c_i$, and so we are free to increase the constants as need be throughout the proof.

Let $D$ be the event that the discrepancy condition fails and let $D(A,B)$ be the event that the discrepancy condition fails for sets $A$ and $B$. Then by the union bound
\begin{eqnarray*}
\prob(D)
&\leq& \prob(\exists A,B \text{ with $|A| \wedge |B|\geq n/e$}): D(A,B) \text{ occurs})\\
&+& \prob(\exists A,B \text{ with $|A| \vee |B| \geq n/e \geq |A| \wedge |B|$}): D(A,B) \text{ occurs})\\
&+&\sum_{A,B: \ |A| \vee |B|< n/e}\prob(D(A,B))
\end{eqnarray*}
Taking $\csubone > e^2$, then
when $|A| \wedge |B| \geq \tfrac{n}{e},$
\[
e(A,B)>\csubone \mu(A,B)>\csubone (n/e)^2d/n>nd.
\]
Thus, there are at least $nd$ edges in the graph. The distribution of the number of edges is binomial with mean $n(n-1)p/2=nd/2$, and so the probability of this is going to zero exponentially in $nd$, i.e.
\begin{equation} \label{verystupid}
\prob(\exists A,B \text{ with $|A| \wedge |B|\geq n/e$}): D(A,B) \text{ occurs})= O(\exp(-cnd))
\end{equation}
for some absolute constant $c >0.$

If $|A| \vee |B| \geq \tfrac{n}{e}> |A| \wedge |B|,$
%then we have that
%\[
%\mu(A,B,n) \geq \frac{|A||B|d}{100n}.
%\]
and if the bounded degree condition holds, then
$e(A,B) \leq (|A| \vee |B|)\csubzero d$ and
\[
\frac{e(A,B)}{\mu(A,B,n)} \leq \frac{\csubzero nd(|A| \vee |B|)}{|A||B|d} = \frac{\csubzero n}{|A| \wedge |B|} \leq \csubzero e.%\leq \csubone .
\]
Thus taking $\csubone > \csubzero e$, we have that
\begin{eqnarray}
\prob(\exists A,B \text{ with $|A| \vee |B| \geq n/e \geq |A| \wedge |B|$}): D(A,B) \text{occurs})
&\leq& \prob(\bdc \text{fails}) \nonumber\\
&=&O(\exp(-md)). \label{twoam}
\end{eqnarray}

Now we need to deal with the case that both $A$ and $B$ are less than $\tfrac{n}{e},$ but at least one is greater than $d^{1/4}/100.$  Take $\csubtwo > 18 + 1200m.$  For emphasis, we will write $\mu(A,B,n)=\mu(A,B) = \frac{|A||B|d}{n}.$ Choose $r=r(A,B,n)=\csubone  \vee r_1$ where $r_1$ is the solution to
\[
\mu(A,B,n) r_1\log(r_1) = \csubtwo(|A| \vee |B|)\log \tfrac{n}{|A| \vee |B|}.
\]

For any $A$, $B$  and $n$ we must have either
\begin{itemize}
\item $e(A,B) \leq r \mu (A,B,n)$ and $r =\csubone $
\item $e(A,B) \leq r \mu (A,B,n)$ and $r =r_1$ or
\item $e(A,B)>r \mu(A,B,n)$
\end{itemize}
Thus if $D(A,B)$ occurs then at least one of the following three events occur.
\begin{itemize}
\item $D_1=D_1(A,B)=\bigg\{e(A,B)\leq r \mu(A,B,n),\,  r=\csubone  \text{ and }\\
{} \hspace{1.9in} e(A,B)> \csubone  \mu(|A|,|B|,n) \bigg\}$
\item $D_2=D_2(A,B)=\bigg\{e(A,B)\leq r{\mu(A,B,n)}, r=r_1 \text{ and } \\
{} \hspace{1.9in} e(A,B)\log \tfrac{e(A,B)}{\mu(A,B,n)}>
\csubtwo (|A| \vee |B|)\log \tfrac{n}{|A| \vee |B|}\bigg\}$
\item $D_3=D_3(A,B)=\{e(A,B)>r \mu(A,B,n)\}$
\end{itemize}

%We will show that events $D_1$ and $D_2$ are empty.
For $D_1$ the conditions are mutually exclusive as $e(A,B)$ can not be
simultaneously greater than and less than or equal to $\csubone \mu(A,B,n)$. Thus $D_1(A,B)$ is empty.
For $D_2$ we get similar contradiction after a little work.
\begin{eqnarray*}
e(A,B) \log\tfrac{e(A,B)}{\mu(A,B,n)}
&>&\csubtwo (|A|\vee |B|)\log \tfrac{n}{|A| \vee |B|}\\
e(A,B) \log\tfrac{e(A,B)}{\mu(A,B,n)}
&>& \mu(A,B,n)r_1\log r_1 \\
\tfrac{e(A,B)}{\mu(A,B,n)} \log\tfrac{e(A,B)}{\mu(A,B,n)}
&>& r_1 \log r_1\\
\tfrac{e(A,B)}{\mu(A,B,n)}  &>& r_1 \\
e(A,B)  &>& r_1 \mu(A,B,n) \\
e(A,B)  &>& r \mu(A,B,n).
\end{eqnarray*}
This is a contradiction so $D_2(A,B)$ is also empty.

Now we bound $\prob(D_3(A,B)).$
%We have that
%$e(A,B)$ is a binomial random variable with mean $\mu=|A|\cdot |B|d/n$.
As $e(A,B)$ is binomial with mean at most $\mu(A,B,n)$, Lemma \ref{maybe} implies
\[
\prob(D_3(A,B)) \leq \exp\left(-\tfrac{\mu(|A|,|B|,n) r \log r}{3}\right)
\]
for any $r\geq 4.$

For all $A,B$ we have $D \subset D_1 \cup D_2 \cup D_3$ and $\prob(D_1(A,B))=\prob(D_2(A,B))=0$. Combining this with (\ref{verystupid}) and (\ref{twoam}) we get
\begin{eqnarray*}
\prob(D)
%\sum_{A,B}\prob(D(A,B))
& \leq & \prob(\exists A,B:\ D(A,B) \text{ occurs})\\
& \leq & \prob(\exists A,B:\ |A|,|B|<n/e \text{ and } D(A,B) \text{ occurs})+O(\exp(-md))\\
& \leq & \prob(\exists A,B:\ |A|,|B|<n/e \text{ and } D_3(A,B) \text{ occurs})+O(\exp(-md))\\
& \leq &
\sum_{|A|,|B|}\prob(D_3(A,B)) +O(\exp(-md)) \\
& \leq &
\sum_{a,b} \sum_{|A|=a,|B|=b}
\exp\left(-\tfrac{\mu r \log r}{3}\right) +O(\exp(-md))\\
& \leq & \sum_{a,b}
{n \choose a} {n\choose b}\exp\left(-\tfrac{\mu(a,b,n) r \log r}{3}\right) +O(\exp(-md)),
\end{eqnarray*}
where the sums are over all pairs $(a,b)$ with $d^{1/4}/100 \leq a \vee b \leq n/e.$
%In what follows, we will think of $a$ and $b$ as being the sizes of $|A|$ and $|B|,$ in preparation to use a union bound.  Let $k = k(a,b,n)$ be defined as the solution of
%\[
%k\log k  = \frac{3(6+m)(a \vee b)n}{abd} \log \frac{n}{a \vee b},
%\]
%or $\csubzero e \vee 4,$ whichever is larger.
To evaluate the last term we get
\begin{eqnarray*}
\tfrac{\mu r \log r}{3}
& \geq &\bigg(6 + 400m\bigg)\bigg((|A| \vee |B|) \log\tfrac{n}{|A| \vee |B|}\bigg)\\
%& \geq &
%|A|2\log\tfrac{n}{|A|}
%+|B|2\log\tfrac{n}{|B|}
%+(2 + m)(|A| \vee |B|) \log\tfrac{n}{|A| \vee |B|}, \\
%\intertext{where we have used the monotonicity of $x\log \tfrac{n}{x}$ on $[1,\tfrac{n}{e}]$,}
& > &
\bigg(2+2+2+(400m))\bigg)\bigg((|A|\vee |B|) \log\tfrac{n}{|A|\vee|B|}\bigg)\\
%+2(|A|\vee|B|)(1 + \log\tfrac{n}{|A|\vee|B|})
%+(2 + m)(|A| \vee |B|) \log\tfrac{n}{|A| \vee |B|}\\
%& \geq &
%2(|A|\vee |B|)(1 + \log\tfrac{n}{|A|\vee|B|})
%+2(|A|\vee|B|)(1 + \log\tfrac{n}{|A|\vee|B|})
%+(2 + m)\log n.\\
& >&
2|A|( \log\tfrac{n}{|A|})
+2|B|( \log\tfrac{n}{|B|})
+2\log n
+4md^{1/4}\log\tfrac{100n}{d^{1/4}}
\\
& >&
|A|(1 + \log\tfrac{n}{|A|})
+|B|(1 + \log\tfrac{n}{|B|})
+2\log n
+3md^{1/4}\log n.
\end{eqnarray*}
The first line is due to the definitions of $r$ and $\csubtwo$. In the third line we use the monotonicity of $x \log \tfrac{n}{x}$ on $[1,n/e]$ by
substituting in $|A|$, $|B|,$ $1$ and $d^{1/4}/100$ for $x$. In the fourth line we use that
$|A| \vee |B| \leq \tfrac{n}{e}$  so $\log \tfrac{n}{|A|},\log \tfrac{n}{|B|}>1$

Exponentiating we get
\[
\exp\left[
\tfrac{\mu r \log r}{3}
\right]
\geq \left(\tfrac{en}{|A|}\right)^n
\left(\tfrac{en}{|B|}\right)^n
n^2 \exp(3md^{1/4}(\log n))
\]
It follows that
%\begin{align*}
%&\prob \left[
%\exists A,B~\text{with}~|A|=a,~|B|=b,~\text{so that}~e(A,B) \geq k(a,b)\mu(A,B)
%\%right] \\
%&\hspace{2in}
\begin{eqnarray*}
{n \choose a} {n\choose b}\exp\left(-\tfrac{\mu(a,b,n) r \log r}{3}\right)
&\leq & {n \choose a}{n \choose b}\left(\tfrac{en}{a}\right)^{-n}
\left(\tfrac{en}{b}\right)^{-n} n^{-2}\exp(-3md^{1/4}\log n)\\
&\leq & n^{-2}\exp(-3md^{1/4}\log n).
\end{eqnarray*}%\end{align*}

Putting this together we get
\begin{eqnarray*}
%\sum_{A,B}\prob(D(A,B))
\prob(D)
& \leq &
\sum_{d/100 \leq a \vee b \leq n/e }{n \choose a} {n\choose b}\exp\left(-\tfrac{\mu(a,b,n) r \log r}{3}\right) +O(\exp(-md))\\
& \leq &
n^2n^{-2}\exp(-3md^{1/4}\log n)+O(\exp(-md)).
\end{eqnarray*}
%
%It remains to show that for all pairs of sets $(A,B)$ with $|A| \vee |B| \leq d^{1/4}/100,$ there is a constant $\csubthree=\csubthree(\delta,m)$ so that $e(A,B) \leq \csubthree d$ with probability at least $1-O(\exp(-md)).$  Note that for such a set
%\[
%\Exp e(A,B) \leq \mu(A,B,n) \leq \frac{d^3}{100^2n},
%\]
%and hence
%\[
%\frac{d}{\Exp e(A,B)} \geq \frac{100^2n}{d^2}.
%\]
%Thus, we have by the standard Chernoff inequality that
%\[
%\Pr\left[
%e(A,B) \geq rd
%\right]
%\leq \exp(-crd)
%\]
%
Thus the lemma is satisfied.\end{pfofthm}

We finally give a quick sketch of how Proposition~\ref{prop:adj} follows from Lemma~\ref{lem:regularity}.  This is nearly the same as Theorem 2.5 of~\cite{FeigeOfek}, and so we will cite heavily.
\begin{pfofthm}{ Proposition \ref{prop:adj}}

We recall that we wish to bound
\[
\sup_{ \substack{ \|x\| = 1, x^t \one = 0 \\ \|y\| = 1 } } | x^tAy| \leq C\sqrt{d}.
\]
For this we will relax the supremum to a finite, discrete space.
Define
\[
\mathcal{U}=\left\{ \frac{ z}{2\sqrt{n}} ~:~ z \in \Z^n, \|z\|^2 \leq 4n\right\}
~~~~~\text{ and }~~~~~
\mathcal{T}=\left\{ z \in \mathcal{U} ~:~ z \perp \one \right\}.
\]
As $\mathcal{U}$ is $\tfrac{1}{2}$-net of the sphere, and $S=\{x~:\|x\| = 1, x^t \one = 0\}$ is in the convex hull of $\mathcal{T}$ (by Lemma 2.3 of~\cite{FeigeOfek}), we have that
\[
\sup_{ \substack{ \|x\| = 1, x^t \one = 0 \\ \|y\| = 1 } } | x^tAy| \leq
4
\sup_{ \substack{ x \in \mathcal{T} \\ y \in \mathcal{U}} } | x^tAy|.
\]
Further, we have that $|\mathcal{T}| \leq |\mathcal{U}| \leq C^n$ for some absolute constant $C.$

For a fixed pair of vectors $(x,y) \in \mathcal{T} \times \mathcal{U},$ define the {\bf light couples} $\mathcal{L} = \mathcal{L}(x,y)$ to be all those ordered pairs $(u,v) \in \{1,2,\dots,n\}^2$ so that $|x_uy_v| \leq \tfrac{\sqrt{d}}{n},$ and let the {\bf heavy couples} $\mathcal{H}=\mathcal{H}(x,y)$ be all those pairs that are not light. We will use the notation
\[
\lc =
\sum_{ (u,v) \in \mathcal{L}} x_uA_{uv}y_v,
\]
and the notation
\[
\hc =
\sum_{ (u,v) \in \mathcal{H}} x_uA_{uv}y_v,
\]

For the light couples, we recall Bernstein's inequality, which says that for independent, centered random variables $\left\{ X_i \right\}_1^N$ such that $|X_i| \leq M$ almost surely for all $1 \leq i \leq N$ and all $t \geq 0,$
  \[
    \Pr\left[ \sum_{i=1}^N X_i > t \right] \leq
    \exp\left( 
    \frac{-t^2}{2\sum_{i=1}^N \Exp X_i^2 + \frac{2}{3}Mt}
    \right).
  \]
  To realize $\lc$ as a sum of independent variables, we need to account for the symmetry in $A.$  Let $N$ be the number of undirected edges $\left\{ u,v \right\}$ so that either $\left( u, v\right)$ or $\left( v,u \right)$ appear in $\mathcal{L}.$  Enumerate these edges and define for $i$ with $1\leq i \leq N$ corresponding to $\left\{ u,v \right\},$
\[
  X_i = 
  (A_{uv}-p)x_uy_v \one[\left( u,v \right) \in \mathcal{L}]
  +(A_{uv}-p)x_vy_u \one[\left( v,u \right) \in \mathcal{L}].
\]
For our purposes, it will be enough to use the bound
\[
\sum_{i=1}^N \Exp X_i^2 \leq
\sum_{i=1}^N 2p \left\{ (x_uy_v)^2 + (x_vy_u)^2 \right\}
\leq 2p \sum_{(u,v)} x_u^2y_v^2 \leq 2p,
\]
where we have used the normalization of the vectors.
In summary, by Bernstein's inequality,
\[
    \Pr\left[ |\lc - \Exp \lc| > t \right] \leq
    \exp\left( 
    \frac{-nt^2}{4d + \frac{2}{3}\sqrt{d}t}
    \right).
\]
To control the expectation, note that on account of $x \in \mathcal{T},$
\[
  \Exp \lc + \Exp \hc = 0
\]
However,
\[
  |\Exp \hc| \leq 
\sum_{ (u,v) \in \mathcal{H}} p|x_uy_v|
\leq
\sum_{ (u,v) \in \mathcal{H}} 
\frac{np}{\sqrt{d}}
|x_uy_v|^2
\leq \sqrt{d}.
\]
As $\mathcal{T}$ is only of cardinality $e^{O(N)},$ for each $m$ there is a constant $C=C(m)$ so that
\[
\Pr \left[
\sup_{(x,y) \in \mathcal{T} \times \mathcal{U}} |\lc| > C\sqrt{d}
\right] \leq Ce^{-mn}.
\]

To control the heavy couples, we use the discrepancy property (c.f. Corollary 2.11 of~\cite{FeigeOfek} or Section 2.3 of~\cite{FKSz}).  The proof is nearly identical to either of those two claims, although it is not exactly either one, on account of the slightly altered definition of discrepancy.
\begin{lemma}
\label{lem:discrepancy}
Suppose $c_1,c_2,C_1$ are constants greater than $1$ and $d>0.$
There is a constant $C>0$ depending only on $c_1,c_2,C_1$ 
so that for any graph with the property that all degrees are bounded by $C_1d$ and for all subsets $A$ and $B$ of vertices
\begin{enumerate}
\item
$
\tfrac{e(A,B)}{\mu(A,B)} \leq \csubone
$
\item
$
e(A,B)\log \tfrac{e(A,B)}{\mu(A,B)} \leq \csubtwo( |A| \vee |B|)\log \tfrac{n}{|A| \vee |B|}
$
\item
$
|A| \vee |B| \leq d^{1/4}/100
$
\end{enumerate}
then for all $x,y \in \mathcal{U}$
\[
 \sum_{ \{u,v\} \in \mathcal{H}} \left|x_uA_{u,v}y_v\right|  \leq C\sqrt{d}.
\]
\end{lemma}

By Lemma~\ref{lem:regularity}, all these conditions hold with the desired probability, and hence the proof of Proposition \ref{prop:adj} is complete.
\end{pfofthm}

\begin{pfofthm}{Lemma~\ref{lem:discrepancy}}
  We will partition the summands into blocks where each term $x_u$ or $y_v$ has approximately the same magnitude.  Let $\gamma_i = 2^{i},$ $n^* = \lceil  \log_2\sqrt{n} \rceil$ and put
\begin{align*}
A_i &= \left\{ u~\big\vert~ \tfrac{\gamma_{i-1}}{\sqrt n} \leq |x_u| < \tfrac{\gamma_{i}}{\sqrt n} \right\},& & 0 \leq i \leq n^*. \\
B_i &= \left\{ u~ \big\vert~ \tfrac{\gamma_{i-1}}{\sqrt n} \leq |y_u| < \tfrac{\gamma_{i}}{\sqrt n} \right\},& & 0 \leq i \leq n^*.
\end{align*}
Let $\hat {\mathcal{H}}$ denote those pairs $(i,j)$ so that $\gamma_i\gamma_j \geq \sqrt{d}.$  The contribution of the absolute sum can, in these terms, be bounded by
\[
 \sum_{ (u,v) \in \mathcal{H}} \left|x_uA_{u,v}y_v\right|  \leq 
 \sum_{ (i,j) \in \hat{\mathcal{H}}} \frac{\gamma_i\gamma_j}{n} e(A_i,B_j). 
\]
%Let $\lambda_{i,j} = \tfrac{e(A_i,B_j)}{\mu(A_i,B_j)}$ denote the discrepancy, which can be controlled using Lemma~\ref{lem:regularity}.  In terms of this quantity, the bound becomes
%\[
% \sum_{ (u,v) \in \mathcal{H}} \left|x_uA_{u,v}y_v\right|  \leq
%   \sum_{ (i,j) \in \hat{\mathcal{H}}} \frac{\gamma_i\gamma_j}{n} \lambda_{i,j} |A_i||B_j|\tfrac{d}{n}. 
%\]
%In this form, the magnitudes of each of the quantities are somewhat
%opaque.  Consider the sum $\sum_{i} |A_i| \frac{\gamma_i^2}{n};$ it is
%at most $4\|x\|^2= 4.$  In particular, it is of constant order.  Thus let $\alpha_i =|A_i| \frac{\gamma_i^2}{n}$ and $\beta_j = |B_j| \frac{\gamma_j^2}{n}.$  This allows the bound to be rewritten as
%\[
%d \sum_{ (i,j) \in \hat{\mathcal{H}}} \frac{\gamma_i^2|A_i|}{n}\frac{\gamma_j^2|B_j|}{n} \frac{\lambda_{i,j}}{\gamma_i\gamma_j} = 
%\tfrac{d}{\sqrt{d}} \sum_{ (i,j) \in \hat{\mathcal{H}}} \alpha_i \beta_j \frac{\lambda_{i,j}\sqrt{d}}{\gamma_i\gamma_j}.
%\]
%This exposes the quantity $\sigma_{i,j} =
%\frac{\lambda_{i,j}\sqrt{d}}{\gamma_i\gamma_j}$ as having some special
%importance.  In effect, we will show that either for fixed $i,$
%$\sum_{j} \sigma_{i,j} \beta_j$ has constant order,  or for fixed $j$, $\sum_{i} \sigma_{i,j}\alpha_i$ has constant order.
%
In what follows, we will bound the contribution of the summands where $|A_i| \geq |B_j|.$  By symmetry, the contribution of the other summands will have the same bound.  The heavy couples will now be partitioned into $6$ classes $\{ \hat{\mathcal{H}}_i\}_{i=1}^6$ where their contribution is bounded in a different way.  Let $\hat{\mathcal{H}}_i \subseteq \hat{\mathcal{H}}$ be those pairs $(i,j)$ which satisfy the $i^{th}$ property from the following list but none of the prior properties:
\begin{enumerate}
  \item $|A_i| < d^{1/4}/100.$
  \item $\tfrac{e(A_i,B_j)}{\mu(A_i,B_j)} \leq c_1 \frac{\gamma_i\gamma_j}{\sqrt{d}}.$
  \item $\gamma_j > \tfrac14 \sqrt{d}\gamma_i.$
  \item $\log \tfrac{e(A_i,B_j)}{\mu(A_i,B_j)} > \frac{1}{2}\log \frac{n}{|A_i|}.$
  \item $\frac{n}{|A_i|} > \gamma_i^4.$
  \item $\frac{n}{|A_i|} \leq \gamma_i^4.$
\end{enumerate}

\subsubsection*{ Bounding the contribution of $\hat{\mathcal{H}}_1$ }
For these terms, we have that $e(A_i,B_j) \leq |A_i||B_j| \leq \frac{\sqrt{d}}{10000}.$  Hence
\[
\sum_{ (i,j) \in \hat{\mathcal{H}}_1} \frac{\gamma_i\gamma_j}{n} e(A_i,B_j)
\leq
\sum_{ i,j=0}^{n^*} \frac{\gamma_i\gamma_j}{n}\frac{\sqrt{d}}{10000}
\leq \frac{16\sqrt{d}}{10000},
\]
where in the last line we have used that $\sum_{i=0}^{ n^*} 2^i \leq 4\sqrt{n}.$

\subsubsection*{ Bounding the contribution of $\hat{\mathcal{H}}_2$} 
Applying the bound directly to the sum, we have that
\[
\sum_{ (i,j) \in \hat{\mathcal{H}}_2} \frac{\gamma_i\gamma_j}{n} e(A_i,B_j)
\leq c_1 \sum_{ (i,j) \in \hat{\mathcal{H}}_2} \frac{\gamma_i^2\gamma_j^2}{n\sqrt{d}} \mu(A_i,B_j)
= c_1\sqrt{d} \sum_{ (i,j) \in \hat{\mathcal{H}}_2} \frac{\gamma_i^2\gamma_j^2}{n}\frac{|A_i||B_j|}{n}.
\]
Further,
\[
  \sum_{i=0}^{n^*} \frac{\gamma_i^2|A_i|}{n}
  \leq 4\sum_{u=1}^n |x_u|^2 \leq 4,
\]
and the same bound holds for the sum over $|B_j|.$ Hence
\[
\sum_{ (i,j) \in \hat{\mathcal{H}}_2} \frac{\gamma_i\gamma_j}{n} e(A_i,B_j)
\leq c_1\sqrt{d} \sum_{i,j=0}^{n^*} \frac{\gamma_i^2\gamma_j^2}{n}\frac{|A_i||B_j|}{n} = 16c_1\sqrt{d}.
\]

\subsubsection*{ Bounding the contribution of $\hat{\mathcal{H}}_3$.}
By the bound on the degrees, we have that $e(A_i,B_j) \leq C_1 |B_j| d.$  Hence
\[
\sum_{ (i,j) \in \hat{\mathcal{H}}_3} \frac{\gamma_i\gamma_j}{n} e(A_i,B_j)
\leq C_1d \sum_{ (i,j) \in \hat{\mathcal{H}}_3} \frac{\gamma_i\gamma_j}{n} |B_j|.
\]
Since $\gamma_i < 4\gamma_j/\sqrt{d},$ upon summing over all possible $i$, we get that for fixed $j$
\[
  \sum_{i :(i,j) \in \hat{\mathcal{H}}_3} \gamma_i \leq \frac{8\gamma_j}{\sqrt{d}}.
\]
Therefore,
\[
\sum_{ (i,j) \in \hat{\mathcal{H}}_3} \frac{\gamma_i\gamma_j}{n} e(A_i,B_j)
\leq C_1\sqrt{d} \sum_{j=0}^{n^*} \frac{8\gamma_j^2}{n} |B_j|
\leq 32C_1\sqrt{d}.
\]

\subsubsection*{ Bounding the contribution of $\hat{\mathcal{H}}_4$.}

As we are not in $\hat{\mathcal{H}}_1$ or $\hat{\mathcal{H}}_2,$ it must be that $(i,j) \in \hat{\mathcal{H}}_4$ satisfy the second discrepancy condition, that is
\[
  \tfrac{1}{2}
  e(A_i,B_j)\log \tfrac{n}{|A_i|} \leq
  e(A_i,B_j)\log \tfrac{e(A_i,B_j)}{\mu(A_i,B_j)} \leq 
  \csubtwo |A_i|\log \tfrac{n}{|A_i|}.
\]
Hence, applying this bound and summing over all $j$ so that $\gamma_j \leq \frac{1}{4}\sqrt{d} \gamma_i,$
\[
\sum_{ (i,j) \in \hat{\mathcal{H}}_4} \frac{\gamma_i\gamma_j}{n} e(A_i,B_j)
\leq
{\csubtwo} \sqrt{d}
%\frac{c_2}{2} \sqrt{d}
\sum_{ i=0}^{n*} {\gamma_i^2} \frac{|A_i|}{n}
\leq 4\csubtwo \sqrt{d}.
\]

\subsubsection*{ Bounding the contribution of $\hat{\mathcal{H}}_5$.}

For $(i,j) \in \hat{\mathcal{H}}_5$ we have
\[
  e(A_i,B_j)
  \leq \mu(A_i,B_j) \left(\tfrac{n}{|A_i|}\right)^{1/2}
  = d{|B_j|} \left(\tfrac{n}{|A_i|}\right)^{-1/2}
  \leq d{|B_j|} \gamma_i^{-2} 
\]
Hence,
\[
\sum_{ (i,j) \in \hat{\mathcal{H}}_5} \frac{\gamma_i\gamma_j}{n} e(A_i,B_j)
\leq
\sum_{ (i,j) \in \hat{\mathcal{H}}_5} \frac{d \gamma_j^2|B_j|}{n\gamma_i\gamma_j} 
\leq
\frac{2}{\sqrt{d}}
\sum_{j=0}^{n^*}
 \frac{d \gamma_j^2|B_j|}{n}
\leq 8\sqrt{d},
\]
where we have used in the penultimate bound that the sum over $i$ is dominated by the series
\[
  \sum_{i : \sqrt{d} \leq \gamma_j\gamma_i} \frac{1}{\gamma_i} \leq \frac{2\gamma_j}{\sqrt{d}}.
\]

\subsubsection*{ Bounding the contribution of $\hat{\mathcal{H}}_6$.}

For $(i,j) \in \hat{\mathcal{H}}_6,$ we have that
\[
  e(A_i,B_j)\log \tfrac{c_1\gamma_i\gamma_j}{\sqrt{d}} \leq
  e(A_i,B_j)\log \tfrac{e(A_i,B_j)}{\mu(A_i,B_j)} \leq 
  \csubtwo |A_i|\log \tfrac{n}{|A_i|} \leq
  4\csubtwo |A_i| \log \gamma_i
\]
This brings us to the bound
\[
\sum_{ (i,j) \in \hat{\mathcal{H}}_6} \frac{\gamma_i\gamma_j}{n} e(A_i,B_j)
\leq
{4c_2}%{\log c_1}
\cdot
\sum_{ (i,j) \in \hat{\mathcal{H}}_6} 
\frac{\gamma_i|A_i|\log \gamma_i }{n}\frac{\gamma_j}{\log(c_1\gamma_i\gamma_j) - \log \sqrt{d}}.
\]
The sum in $j$ only runs over those terms such that $4\gamma_j \leq \sqrt{d}\gamma_i$ and such that $\gamma_j\gamma_i \geq \sqrt{d}.$  For $j$ such that $\gamma_j \leq \gamma_i\sqrt{d}/(1+\log(\gamma_i))$ we bound the sum over $j$ by
\[
  \sum_{j} \frac{\gamma_j}{\log(c_1\gamma_i\gamma_j) - \log \sqrt{d}}
  \leq
  \sum_{j} \frac{\gamma_j}{\log c_1}
  \leq  \frac{
 2\gamma_i\sqrt{d} 
  }{
 (\log c_1)(1+\log \gamma_i)
 }.
\]
For larger $j$, we bound the sum by
\[
  \sum_{j} \frac{\gamma_j}{\log(c_1\gamma_i\gamma_j) - \log \sqrt{d}}
  \leq
  \sum_{j} \frac{\gamma_j}{\log c_1\gamma_i^2 -\log(1+\log \gamma_i)}
  \leq  \frac{
 \gamma_i\sqrt{d} 
  }{
 2(\log c_1)(\log \gamma_i)
 },
\]
having applied the inequality $\log(1+x)\leq x.$
Hence, we conclude that 
\[
\sum_{ (i,j) \in \hat{\mathcal{H}}_6} \frac{\gamma_i\gamma_j}{n} e(A_i,B_j)
\leq
\frac{10 c_2 \sqrt{d}}{\log c_1}
\cdot
\sum_{ (i,j) \in \hat{\mathcal{H}}_6} 
\frac{\gamma_i^2|A_i| }{n}
\leq \frac{40 c_2 \sqrt{d}}{\log c_1}.
\]
%\subsubsection{ Assembling the bound }
%We must sum the contributions of each of the classes of couples.   Recall that we must double the contribution here because we have only considered couples where $|A_i| \geq |B_j|.$  In each of the cases outlined above, it only remains to sum over the $\alpha_i$ or $\beta_j$ in each bound.  Doing so contributes a factor of $4$ to each bound, so that the constant can be given by 
%\[
%2\left[
%16c_1+32+8c_2+
%\frac{32c_2}{c_1^2}
%+8
%\right]
%\]
%\end{proof}
%
%and it shows that there is a constant $C=C(\delta,m)$ sufficiently large so that
%\begin{multline*}
%\Pr \left[
%\sup_{(x,y) \in \mathcal{T} \times \mathcal{U}} |\hc|  > C\sqrt{d}
%\right] \\
%\leq
%\Pr \left[ \bdc \text{ fails } \right]
%+
%\Pr \left[ \text{ discrepancy fails } \right].
%\end{multline*}
%
\end{pfofthm}

%
%\section{Remarks}
%
%Our main theorem parallels earlier results, in particular Theorem 4 of \cite{Zuk}.  For the ``triangular model'' of random group, \.Zuk showed that $d=1/3$ is the threshold for property~(T).  The threshold for property~(T) in Gromov's density random groups is still not known --- \.Zuk showed that groups at density $> 1/3$ have property~(T), and Ollivier and Wise showed that groups at density $< 1/5$ do not \cite{OW11}.

%This work provides examples of hyperbolic groups with property~(T) \cite{Ollivier}.
%
%Combining with \cite{bhk11}, the results here also give random examples of hyperbolic groups with property~(T), for $\pi_1(Y)$ with $Y \sim Y_2(n,p)$ and
%$$ \frac{(2 + \epsilon )\log n }{n} \le p \le \frac{1}{n^{1/2 + \epsilon}}.$$
%The random fundamental group has also been studied by Costa and Farber, who showed that if
%$$ \frac{3 + \epsilon  }{n} \le p \le \frac{1}{n^{46/47 + \epsilon}},$$
%then w.h.p.\ $\pi_1(Y)$ has cohomological dimension $2$ \cite{CF12}, and in particular that $\pi_1(Y)$ is an infinite group.

  %In proving \.Zuk made use of earlier results on spectral gap of random $d$-regular graphs due to Friedman \cite{Friedman}.  Coja-Oghlan has a result (Theorem 1.2 in \cite{CojaOghlan}) analogous to our Theorem \ref{thm:ergap1}, but requires $p$ to be slightly larger $p \ge C \log n / n$ for some sufficiently large constant $C$.

\begin{appendices}
\section{Estimates of Binomial Random Variables}
%\begin{lemma}
%\label{lem:binomialestimates}
%Let $X \sim \operatorname{Binomial}(n,p),$ then it follows that
%\[
%\frac{1}{\sqrt{np + 1}} \leq \expect\left[\frac{1}{\sqrt{1+X}} \right]
%\leq \frac{1}{\sqrt{p(n+1)}}.
%\]
%\end{lemma}
%
%\begin{pfofthm}{Lemma \ref{lem:binomialestimates}}
%The lower bound follows as an immediate consequence of Jensen's inequality.  For the upper bound, applying Jensen gives that
%\[
%\expect\left[\frac{1}{\sqrt{1+X}} \right] \leq
%\sqrt{\expect\left[\frac{1}{{1+X}} \right]}.
%\]
%The proof will be completed by verifying the identity
%\[
%\expect \left[\frac{1}{X+1}\right] = \frac{1}{p(n+1)}(1-(1-p)^{n+1}),
%\]
%from which the claim follows immediately.
%
%Expand the expectation as a sum
%\begin{align*}
%\expect \left[\frac{1}{X+1}\right]
%&= \sum_{i=0}^n \frac{1}{1+i} {n \choose i} p^{i}(1-p)^{n-i} \\
%\intertext{Apply the binomial coefficient identity $\frac{1}{i+1} {n \choose i} = \frac{1}{n+1}{ n + 1 \choose i+1},$}
%&= \sum_{i=0}^n \frac{1}{n+1}{ n + 1 \choose i+1} p^{i}(1-p)^{n-i} \\
%&= \frac{1}{p(n+1)}\sum_{i=0}^n {n + 1 \choose i+1} p^{i+1}(1-p)^{(n+1)-(i+1)} \\
%&= \frac{1}{p(n+1)}\left( 1- (1-p)^{n+1}\right),
%\end{align*}
%where the last step follows from comparison with the binomial theorem.
%
%\end{pfofthm}
\begin{lemma} \label{magic}
Let $X$ be a binomial random variable with mean $\mu$. Then for any $t \leq \mu$
\[
\prob \left[
X \leq t
\right] \leq \exp\left[
-\mu + t( 1+ \log \tfrac{\mu}{t})
\right],
\]
\end{lemma}

\begin{pfofthm}{Lemma~\ref{magic}}
The proof follows from a standard estimate on the Laplace transform combined with Markov's inequality.  For any $\lambda \in \R,$ the Laplace transform of $X \sim \operatorname{Binomial}(n,p)$ can be bounded by
\begin{align*}
\expect e^{\lambda X}
&= \left(pe^{\lambda} + (1-p)\right)^n \\
&= \left(1 + p(e^{\lambda}-1)\right)^n \\
&\leq \exp\left[\mu(e^{\lambda}-1)\right].
\end{align*}
Provided that $\lambda < 0,$ the tail bound now can be bounded by Markov's inequality by
\begin{align*}
\prob \left[
X \leq t
\right] &=
\prob \left[
e^{\lambda X} \geq e^{\lambda t}
\right] \\
&\leq \left[\expect e^{\lambda X}\right] e^{-\lambda t} \\
&\leq \exp\left[\mu(e^{\lambda}-1) - \lambda t\right].
\end{align*}
Assuming that $t < \mu,$ this bound holds with $\lambda=\log (t/\mu),$ which upon evaluation gives
\[
\prob \left[
X \leq t
\right]
\leq \exp\left[\mu(e^{\log (t/\mu)}-1) - \log (t/\mu) t\right]
= \exp \left[ -\mu + t(1 + \log \tfrac{\mu}{t}) \right].
\]
\end{pfofthm}

\begin{lemma} \label{maybe}
Let $X$ be a binomial random variable with mean $\mu$. Then for any $t>4$
\[
\prob \left[
X \geq t \mu
\right] \leq \exp\left[
-\frac{t \mu \log(t)}{3}
\right],
\]
\end{lemma}

\begin{pfofthm}{Lemma \ref{maybe}}
The proof here is identical in approach to the proof of Lemma~\ref{magic}. As there, it is possible to bound the Laplace transform of $X$ as
\[
\expect e^{\lambda X} \leq
\exp\left[\mu(e^{\lambda}-1)\right],
\]
for any real $\lambda.$  For $\lambda > 0,$ the tail bound follows from Markov's inequality by
\begin{align*}
\prob \left[
X \geq t\mu
\right] &=
\prob \left[
e^{\lambda X} \geq e^{\lambda t\mu}
\right] \\
&\leq \left[\expect e^{\lambda X}\right] e^{-\lambda t\mu} \\
&\leq \exp\left[\mu(e^{\lambda}-1) - \lambda t\mu \right].
\end{align*}
For $t > 1,$ it is possible to take $\lambda = \log t.$  This gives the bound on the tail probability
\[
\prob \left[
X \geq t\mu
\right]
\leq\exp\left[\mu\left(t - 1  -  t \log t\right)\right].
\]
To complete the proof, it remains to show that $t-1 \leq \tfrac{2}{3}t\log t$ when $t \geq 4.$ The function $\frac{t}{t-1}\log t$ is monotonically increasing for $t > 1,$ and thus it suffices to show that $\frac{4}{3}\log 4 \geq \tfrac{3}{2},$ or equivalently that $\log 4 \geq \tfrac{9}{8}.$  This follows from $\log 4 = \int_1^4 \tfrac 1 x dx$ and bounding the integral from below by a right Riemann sum.
\end{pfofthm}

\end{appendices}

\bibliographystyle{plain}
\bibliography{Trefs}

\end{document}